\documentclass[11pt]{article}
\usepackage[centertags]{amsmath}
\usepackage{mathtools}
\usepackage{amssymb}
\usepackage{amsthm}
\usepackage{graphicx}
\usepackage{upquote}
\usepackage{cases}
\usepackage{hyperref}
\usepackage{xcolor}
\usepackage{multicol}
\usepackage[normalem]{ulem}
\usepackage{scalerel}

\usepackage{scalerel}

\newcommand\reallywidehat[1]{\arraycolsep=0pt\relax%
\begin{array}{c}
\stretchto{
  \scaleto{
    \scalerel*[\widthof{\ensuremath{#1}}]{\kern-.5pt\bigwedge\kern-.5pt}
    {\rule[-\textheight/2]{1ex}{\textheight}} 
  }{\textheight} %
}{0.5ex}\\           
#1\\                 
\rule{-1ex}{0ex}
\end{array}
}

\newtheorem{tw}{Theorem}[section]
\newtheorem{lem}[tw]{Lemma}
\newtheorem{cor}[tw]{Corollary}
\newtheorem{prop}[tw]{Proposition}
\newtheorem{definition}{Definition}[section]
\newtheorem{rem}{Remark}
\newtheorem{example}{Example}[section]

\def\X{{\mathcal X}}
\def\G{{\mathcal G}}

\def\N{{\mathbb N}}
\def\t{{\mathbf t}}
\def\s{{\mathbf s}}

\def\={\hspace{-3mm}&=&\hspace{-3mm}}

\title{ 

ON COMPACT ALMOST AUTOMORPHIC FUNCTIONS: NEW CHARACTERIZATIONS, INVARIANCE UNDER CONVOLUTION PRODUCTS AND APPLICATIONS }
\author{ Alan Chávez$^1$\footnote{E-mail: ajchavez@unitru.edu.pe (corresponding author)}~, Jolbyn Casta\~neda $^2$\footnote{E-mail: jocastanedar@unitru.edu.pe}, Alexis R. Carranza $^3$\footnote{E-mail: arodriguezca@unitru.edu.pe}, Kamal Khalil $^4$\footnote{E-mail: kamal.khalil.00@gmail.com.}\\
\small{$^{1,3}$ OASIS research group, Instituto de investigaci\'on en Matem\'aticas, }\\
\small{ Departamento de Matem\'aticas, FCFYM, Universidad Nacional de Trujillo,}\\
\small{Av. Juan Pablo II S/N, Trujillo-Per\'u}\\
\small{$^{2}$ OASIS research group \& Escuela de Matem\'aticas - FCFYM, Universidad Nacional de Trujillo,}\\
\small{Av. Juan Pablo II S/N, Trujillo-Per\'u}\\
\small{$^{4}$ LMAH, University of Le Havre Normandie, FR-CNRS-3335, ISCN, Le Havre 76600, France,}\\\\
\small{(Dedicated to the memory of the late professor Manuel Pinto)}\\
}

\begin{document}

\maketitle

\begin{abstract}
Let \(\mathcal{G}\) be a non-empty subset of the Euclidean space \(\mathbb{R}^m\) (\(m \geq 1\)). This work is dedicated to further investigating the properties of \(\mathcal{G}\)-multi-almost automorphic functions defined on \(\mathbb{R}^m\) with values in a Banach space \(\X\). Using the theory of \(\mathcal{G}\)-multi-almost automorphic functions, we provide two new characterizations of compact almost automorphic functions. In the first characterization, \(\mathcal{G}\) corresponds to the lattice subgroup \(\mathbb{Z}^m \subset \mathbb{R}^m\), while in the second, \(\mathcal{G}\) is taken to be a dense subset of \(\mathbb{R}^m\). Furthermore, we establish the invariance of the space of bounded and compactly \(\mathcal{G}\)-multi-almost automorphic functions under integral operators with Bi-almost automorphic kernels. Finally, we apply these results to the analysis of the almost automorphic dynamics of Poisson's and heat equations. Specifically, we prove the persistence of compact almost automorphic properties in solutions to these equations assumming that the source or initial data, respectively, is \( \mathcal{G} \)-multi-almost automorphic, where \( \mathcal{G} \) is either the lattice subgroup \( \mathbb{Z}^m \) or a dense subset (subgroup) of \( \mathbb{R}^m \).

%

\end{abstract}

\tableofcontents

\section{Introduction}\label{sec:Introduction}
%
%
%


The theory of almost automorphic functions was introduced by Salomon Bochner in 1955–1956 in the context of his work on Differential Geometry \cite{03}. Subsequently, Bochner himself, along with other distinguished mathematicians, further developed the properties of this class of functions in works such as \cite{04,05,18,veech1965almost}. Since then, the theory of almost automorphic functions has attracted sustained attention within the mathematical community. 
This interest arises for several reasons. On the one hand, almost automorphic functions constitute a natural generalization of periodic and almost periodic functions, providing a suitable framework for the analysis of differential, integral, and integro-differential equations, as illustrated in \cite{campos2014almost, campos2020barycentric, rezoug2024asymptotically,WShenYiY2008book,yi2004almost} and related works. On the other hand, the notion of almost automorphy offers a rich setting for the study of abstract topological dynamical systems, as reflected in more recent contributions such as \cite{fuhrmann2020tameness, garcia2021mean, lenz2024abstract, lenz2024pure} and the references cited therein.

The concept of complex-valued almost automorphic functions on a topological group \( G \) was first developed by Veech \cite{18, veech1965almost}. Let \( \mathcal{X} \) be a Banach space with norm $ ||\cdot ||_{\mathcal{X}}$, \( G  \) the Euclidean topological space  \(\mathbb{R}^m \), and let \( \mathrm{R} \) denote the collection of all sequences in \( \mathbb{R}^m \). The notion of almost automorphy for a function \( F: \mathbb{R}^m \to \mathcal{X} \) is described in the following definition

\begin{definition}(Bochner)\label{dfAA}
Let \( F: \mathbb{R}^m \rightarrow \mathcal{X} \) be a continuous function. Then, \( F \) is almost automorphic if and only if, for any sequence \( ( \mathbf{b}_k = (b_k^1, b_k^2, \dots, b_k^m)) \in \mathrm{R} \), there exist a subsequence \( ( \mathbf{b}_{k_l} = (b_{k_l}^1, b_{k_l}^2, \dots, b_{k_l}^m) ) \) of \( \mathbf{b}_k \) and a function \( F^* : \mathbb{R}^m \to \mathcal{X} \) such that the following pointwise limits hold
\[
F^*(\mathbf{t}) = \lim_{l \to \infty} F(\mathbf{t} + \mathbf{b}_{k_l}), \quad
F(\mathbf{t}) = \lim_{l \to \infty} F^*(\mathbf{t} - \mathbf{b}_{k_l}),\, \, \, \t \in \mathbb{R}^m\, .
\]
\end{definition}

Additionally, if the previous limits are replaced by the following ones
\[
\lim_{l \to \infty} \sup_{\mathbf{t} \in \mathcal{K}} \| F(\mathbf{t} + \mathbf{b}_{k_l}) - F^*(\mathbf{t}) \|_{\mathcal{X}} = 0, \quad
\lim_{l \to \infty} \sup_{\mathbf{t} \in \mathcal{K}} \| F^*(\mathbf{t} - \mathbf{b}_{k_l}) - F(\mathbf{t}) \|_{\mathcal{X}} = 0,
\]
where \( \mathcal{K} \) is any compact subset of \( \mathbb{R}^m \), then \( F \) is called compact almost automorphic.

Note that in the provided definition, the sequence \( (\mathbf{b}_k)_{k \in \mathbb{N}} \) can be any sequence in the Euclidean space \( \mathbb{R}^m \). However, a new approach explored in works such as \cite{chavez2022multi, chavez2023almost} involves considering these sequences within a set \( \mathrm{R} \), which is a non-empty collection of sequences in \( \mathbb{R}^m \) (rather than the collection of all sequences in \( \mathbb{R}^m \)). This perspective led the first author and collaborators to introduce the concept of \( \mathrm{R} \)-multi-almost periodic/automorphic functions, a more general notion than the classical concept of almost periodicity/automorphy. In \cite{chavez2022multi, chavez2023almost}, the authors extensively studied the fundamental structural properties of \( (\mathrm{R}, \mathcal{B}) \)-multi-almost periodic/automorphic functions, a concept developed for functions $F:\mathbb{R}^m \times \mathcal{X} \to \mathcal{Y}$ with $\mathcal{Y}$ a Banach space and $\mathcal{B}$ a collection of subsets of $\mathcal{X}$, and provided applications to both partial differential equations and integral equations.

Let \( \mathcal{G}\) be a non empty subset of \(\mathbb{R}^m\). In the present work, we will denote by \( \mathrm{R}_\mathcal{G} \) a non empty collection of sequences in \( \mathcal{G} \), while \( \mathrm{R}_\mathcal{G}^a \) will denote the collection of all sequences in \( \mathcal{G} \). What is involved in this work is the notion of \( \mathrm{R}_\mathcal{G} \)-multi-almost automorphic function, which is as follows (see also Definition \ref{eovako} below)

\begin{definition}\label{eovako00}\index{function!(compactly) $({\mathrm{R}}_{{\mathcal{G}}},{\mathcal B})$-multi-almost automorphic}
Let $F : {\mathbb R}^{m} \rightarrow \mathcal{X}$ be a continuous function.  $F(\cdot)$ is ${\mathrm{R}}_{{\mathcal{G}}}$-multi-almost automorphic if, for every sequence $({\bf b}_{k}=(b_{k}^{1},b_{k}^{2},\cdot \cdot\cdot ,b_{k}^{m})) \in {\mathrm{R}}_{{\mathcal{G}}}$ there exist a subsequence $({\bf b}_{k_{l}}=(b_{k_{l}}^{1},b_{k_{l}}^{2},\cdot \cdot\cdot , b_{k_{l}}^{m}))$ of $({\bf b}_{k})$ and a function
$F^{\ast} : {\mathbb R}^{m} \rightarrow \mathcal{X}$ such that the following pointwise limits hold
\begin{align}\label{love1234567800}
\lim_{l\rightarrow +\infty}F\bigl({\bf t} + {\bf b}_{k_{l}} \bigr)=F^{\ast}({\bf t} ),\, \, \, \t \in \mathbb{R}^m\, ,
\end{align}
and
\begin{align}\label{love12345678900}
\lim_{l\rightarrow +\infty}F^{\ast}\bigl({\bf t} - {\bf b}_{k_{l}} \bigr)=F({\bf t} ),\, \, \, \t \in \mathbb{R}^m\, .
\end{align}
\end{definition}

When \(\mathrm{R}_\mathcal{G}=\mathrm{R}_\mathcal{G}^a \), by way of simplifying things, \( \mathrm{R}_\mathcal{G}^a \)-multi-almost automorphic functions will be referred to as \( \mathcal{G} \)-multi-almost automorphic functions.

The concept of \( \mathcal{G} \)-multi-almost automorphic functions is not entirely new. In the one-dimensional case, where \( \mathcal{G} = \mathbb{Z} \) (the discrete subgroup of \( \mathbb{R} \)), the first author and collaborators introduced the class of piecewise continuous almost automorphic functions (also known as \( \mathbb{Z} \)-almost automorphic functions) between 2013 and 2014, see  \cite{chavez2014discontinuous,chavez2014discontinuous02}. Initially, this class of discontinuous functions was developed to study differential equations with piecewise constant arguments \cite{chavez2014discontinuous}. Later on, the approach was extended to investigate differential equations with deviating arguments, as demonstrated in works such as \cite{ding2017asymptotically, qi2022piecewise}. Let us mention that differential equations with piecewise constant arguments are of hybrid type, combining features of both continuous and discrete dynamical systems, their discrete dynamic is directly related to difference equations. In this way, the study of almost automorphic differential equations with piecewise constant arguments naturally leads to the notion of discrete almost automorphic functions, as well as to the analysis of discrete almost automorphic difference equations. In this context, the theory of discrete almost automorphic functions provides a natural foundation for the theory of $\mathbb{Z}$-almost automorphic functions developed in earlier works such as \cite{chavez2014discontinuous,ding2017asymptotically, qi2022piecewise}, and for the $\mathbb{Z}^m$-multi-almost automorphic functions introduced in the present work (see Section \ref{section char}). For a comprehensive treatment of discrete almost automorphic functions and the study of discrete almost automorphic solutions to linear and nonlinear difference equations, we refer the reader to \cite{araya2009almost} and the references cited therein.

More recently, the theory of \( \mathbb{Z} \)-almost automorphic functions with values in a general Banach space has been explored in \cite{Alanpiecewise}. With the help of this theory, it is proved that  every compact almost automorphic function is \( \mathbb{Z} \)-almost automorphic and uniformly continuous, and vice versa, providing a new characterization of the compact almost automorphic function space. Additionally, in the same work, \( \mathbb{Z} \)-almost automorphic functions were used to offer a straightforward proof of the characterization of discrete almost automorphic functions as restrictions to \( \mathbb{Z} \) of compact almost automorphic functions. Consequently, new applications beyond the theory of differential equations for \( \mathbb{Z} \)-almost automorphic functions have emerged. 

Considering the aforementioned preliminaries, the applications of the class of \( \mathbb{Z} \)-almost automorphic functions and  recognizing the significance of almost automorphic functions in the study of differential equations, this work makes the following three main  contributions
\begin{itemize}
\item First, we introduce the class of $\mathcal{G}$-multi-almost automorphic functions, study its main properties and provide  two new characterizations of the class of compact almost automorphic functions (Sections \ref{section maa} - \ref{section char}). In this way, we expand the theoretical framework established in earlier studies such as \cite{Alanpiecewise,CHAKPPINTO2023} and references cited therein.

\item Second, we obtain invariance results for the new class of almost automorphic function space under convolution products, providing interesting results for integral operators with kernels that depends on two variables (section \ref{section conv}).

\item Third, we analyse and prove the compact almost automorphic regularity of Poisson's and heat equations, generalizing the results presented, for instance, in \cite{CHAKPPINTO2023,sibuya1971almost} (section \ref{section appl}).
\end{itemize}

\noindent Let us develop into more detail regarding the third point mentioned. Let us consider the Poisson equation
\begin{equation}\label{Poisson 0}
\Delta u=f,
\end{equation}
where \( \Delta \) denotes the Laplace operator in the spatial variable \( x = (x_1, \ldots, x_m) \), defined as \( \Delta = \sum_{i=1}^m \partial^2_{x_i} \). After characterizing compact almost automorphic functions, we establish, for equation (\ref{Poisson 0}), the following result: {\it if \( f \) is bounded and \( \mathcal{G} \)-multi-almost automorphic, where \( \mathcal{G} \) is either the lattice subgroup \( \mathbb{Z}^m \) or a dense subset of \( \mathbb{R}^m \), then every bounded and continuous function which is a distributional solution of the Poisson's equation (\ref{Poisson 0}) is compact almost automorphic}.

Now,  \color{blue} {let} \color{black} us consider the following initial value problem for the heat equation:  
\begin{equation}\label{Heat equation 1}
            \left\{
                 \begin{array}{lll}
                \partial_t u(t,x)&=&\Delta u(t,x), \quad t>0,\, x\in \mathbb{R}^m,\\
                   u(0,x) &= &f(x), \quad x\in \mathbb{R}^m,
                 \end{array}\right.
 \end{equation}
where \( \partial_t \) denotes the partial derivative with respect to time \( t \) and \(\Delta \) is the Laplace operator in space variable. Analogously to the Poisson's equation, we prove: {\it if the initial data \( f \) in (\ref{Heat equation 1}) is bounded and \( \mathcal{G} \)-multi-almost automorphic, where \( \mathcal{G} \) is either the lattice subgroup \( \mathbb{Z}^m \) or a dense subset of \( \mathbb{R}^m \); then, for every $t>0$ the solution \( u(t,\cdot) \) of the heat equation (\ref{Heat equation 1}) is compact almost automorphic in space variable}.

Therefore, our work is not only of a theoretical nature, but also includes an applied component to classical partial differential equations. In particular, we emphasize the persistence of compact almost automorphic properties in solutions to the Poisson and heat equations under weaker assumptions than those required for classical almost automorphy on the source term and initial data, respectively. 

\color{blue} 

If the reader is interested in the existence of compact almost automorphic solutions for abstract evolution equations, he or she may consult the work \cite{es2022compact}.\color{black}


The organization of the present work is as follows: in Section \ref{section maa}, we provide the definition and some basic facts about $({\mathrm{R}}_{{\mathcal{G}}},{\mathcal B})$-multi-almost automorphic functions. 
In Section \ref{section dense}, we study $\mathcal{G}$-multi-almost automorphic functions, where \( \mathcal{G} \) is a dense subgroup of \( \mathbb{R}^m \). In Section \ref{section char}, we present two new characterizations of the class of compact almost automorphic function space. In Section \ref{section conv}, we prove the invariance under convolution products with compact Bi-almost automorphic kernels of the bounded and compactly \( \mathrm{R}_{\mathcal{G}} \)-multi-almost automorphic function space. In Section \ref{section appl}, we provide applications to the compact almost automorphic regularity of the Poisson and heat equations. Finally, we present three appendices. Appendix A covers basic results and proofs related to almost automorphic functions on topological groups. Appendix B presents the conditions under which a subgroup of the Euclidean space becomes dense. Appendix C provides a result concerning continuous functions.

\section{On $({\mathrm{R}}_{{\mathcal{G}}},{\mathcal B})$-multi-almost automorphic functions}\label{section maa}

Throughout the present work, $\mathcal{X}, \mathcal{Y}$ denote Banach spaces, we assume that $m\in {\mathbb N}$, ${\mathcal B}$ is a non-empty collection of subsets of $\mathcal{X}$, \( {\mathcal{G}} \) is a non-empty subset of ${\mathbb R}^{m}$, ${\mathrm{R}}_{{\mathcal{G}}}$ is a non-empty collection of sequences in $\mathcal{G}$ and ${\mathrm{R}}_{{\mathcal{G}}}^a$ is the collection of all the sequences in $\mathcal{G}$. Usually, ${\mathcal B}$ denotes the collection of all bounded subsets of $\mathcal{X}$ or all compact subsets of $\mathcal{X}$. We will always assume that for every $x\in \mathcal{X}$, there exist $B\in {\mathcal B}$ such that $x\in B.$ Also, $\mathcal{L}(\X ; \mathcal{Y})$ denote the space of bounded and linear operators from $\X$ to $\mathcal{Y}$, $BUC(\mathbb{R}^m; \X)$ is the space of bounded and uniformly continuous functions from  $\mathbb{R}^m$ to $ \X$, while $ UC(\mathbb{R}^m; \X)$ is the space of uniformly continuous functions from  $\mathbb{R}^m$ to $ \X$ \color{blue} and, $BC(\mathbb{R}^m ; \mathcal{X})$ is the space of bounded continuous functions from $\mathbb{R}^m$ to $\mathcal{X}$\color{black}. 

\subsection{Basic definitions}

\begin{definition}\label{eovako}\index{function!(compactly) $({\mathrm{R}}_{{\mathcal{G}}},{\mathcal B})$-multi-almost automorphic}
Suppose that $F : {\mathbb R}^{m} \times \mathcal{X} \rightarrow \mathcal{Y}$ is a continuous function. Then, we say that $F(\cdot;\cdot)$ is $({\mathrm{R}}_{{\mathcal{G}}},{\mathcal B})$-multi-almost automorphic if, for every $B\in {\mathcal B}$ and for every sequence $({\bf b}_{k}=(b_{k}^{1},b_{k}^{2},\cdot \cdot\cdot ,b_{k}^{m})) \in {\mathrm{R}}_{{\mathcal{G}}}$ there exist a subsequence $({\bf b}_{k_{l}}=(b_{k_{l}}^{1},b_{k_{l}}^{2},\cdot \cdot\cdot , b_{k_{l}}^{m}))$ of $({\bf b}_{k})$ and a function
$F^{\ast} : {\mathbb R}^{m} \times \mathcal{X} \rightarrow \mathcal{Y}$ such that
\begin{align}\label{love12345678}
\lim_{l\rightarrow +\infty}F\bigl({\bf t} + {\bf b}_{k_{l}};x\bigr)=F^{\ast}({\bf t};x) 
\end{align}
and
\begin{align}\label{love123456789}
\lim_{l\rightarrow +\infty}F^{\ast}\bigl({\bf t} -{\bf b}_{k_{l}};x\bigr)=F({\bf t};x),
\end{align}
pointwise for ${\bf t}\in {\mathbb R}^{m}$ and uniformly for $x\in B$. If the above limits converge uniformly in $\t$ on compact subsets of ${\mathbb R}^{m}$, then we say that 
$F(\cdot ; \cdot)$ is compactly $({\mathrm{R}}_{{\mathcal{G}}},{\mathcal B})$-multi-almost automorphic.  
\end{definition}
In the particular case in which ${\mathrm{R}}_{{\mathcal{G}}}={\mathrm{R}}_{{\mathcal{G}}}^a$, we say that $F(\cdot ; \cdot)$ is $({\mathcal{G}},{\mathcal B})$-multi-almost automorphic (resp. compactly $({\mathcal{G}},{\mathcal B})$-multi-almost automorphic). Also note that, if $\mathcal{G}=\mathbb{R}^m$ then, denoting ${\mathrm{R}}_{{\mathcal{G}}}={\mathrm{R}}$, we are in the class of $({\mathrm{R}},{\mathcal B})$-multi-almost automorphic functions studied in \cite{chavez2022multi}.

Also, if the function $F({\bf t};x) =F( {\bf t} )$ (i.e., $F : {\mathbb R}^{m} \rightarrow \mathcal{Y}$), then we say that $F$ is ${\mathrm{R}}_{{\mathcal{G}}}$-multi-almost automorphic (resp. compactly ${\mathrm{R}}_{{\mathcal{G}}}$-multi-almost automorphic) and that space is denoted by $ {{\mathrm{R}}_{{\mathcal{G}}}}AA({\mathbb R}^{m} ; \mathcal{Y})$ (resp. $ \mathcal{K} {\mathrm{R}}_{{\mathcal{G}}}AA({\mathbb R}^{m} ; \mathcal{Y})$). In this situation, if ${\mathrm{R}}_{{\mathcal{G}}}={\mathrm{R}}_{{\mathcal{G}}}^a$, $F(\cdot)$ will be called ${\mathcal{G}}$-multi-almost automorphic (resp. compactly ${\mathcal{G}}$-multi-almost automorphic) and the corresponding space is denoted by $ {{\mathcal{G}}}AA({\mathbb R}^{m} ; \mathcal{Y})$ (resp. $ \mathcal{K} {\mathcal{G}}AA({\mathbb R}^{m} ; \mathcal{Y})$). Also, $ AA({\mathbb R}^{m} ; \mathcal{Y})$ (resp. $ \mathcal{K}AA({\mathbb R}^{m} ; \mathcal{Y})$) will denote the space of (classical) almost automorphic functions (resp. compactly almost automorphic functions).

In the next definitions, we introduce the notions of continuity and of uniform continuity that are relevant to our theoretical results.

 \color{black}
 
 \begin{definition}\label{Def Continuity}
Let $F : {\mathbb R}^{m} \times \mathcal{X} \rightarrow \mathcal{Y}$. Then, the function $F(\cdot;\cdot)$ is called continuous in $\mathbf{t}_0 \in \mathbb{R}^m$, uniformly for any $B\in {\mathcal B}$ if, for every $B\in {\mathcal B}$ and every
$\epsilon>0$, there exists $\delta =\delta (t_0,\epsilon,B)>0$, such that, if $\t \in \mathbb{R}^m$ with $||\t-\t_0||<\delta$, then 
$$
\sup_{x\in B}\| F \bigl({\bf t};x\bigr) - F\bigl({\bf t}_0;x\bigr) \|_{\mathcal{Y}}<\epsilon\, .
$$

$F : {\mathbb R}^{m} \times \mathcal{X} \rightarrow \mathcal{Y}$ is called continuous in $\mathbb{R}^m$, uniformly for any $B\in {\mathcal B}$ if it is continuous at each point of $\mathbb{R}^m$, uniformly for any $B\in {\mathcal B}$.
\end{definition}

\begin{definition}\label{UnifContF}
Let $F : {\mathbb R}^{m} \times \mathcal{X} \rightarrow \mathcal{Y}$. Then, the function $F(\cdot;\cdot)$ is called uniformly continuous in $\mathbb{R}^m$, uniformly for any $B\in {\mathcal B}$ if, for every $B\in {\mathcal B}$ and every
$\epsilon>0$, there exists $\delta =\delta (\epsilon,B)>0$, such that, if $\t,\s \in \mathbb{R}^m$ with $||\t-\s||<\delta$, then 
$$
\sup_{x\in B}\| F \bigl({\bf t};x\bigr) - F\bigl({\bf s};x\bigr) \|_{\mathcal{Y}}<\epsilon\, .
$$
\end{definition}
In this situation, we will say that $F : {\mathbb R}^{m} \times \mathcal{X} \rightarrow \mathcal{Y}$ is uniformly continuous in the first variable.

In the following definition, we introduce the notion of $({\mathrm{R}}_{{\mathcal{G}}},{\mathcal B})$-uniformly continuous function.

\begin{definition}\label{(R,B) uniformly continuous}
Let $F : {\mathbb R}^{m} \times \mathcal{X} \rightarrow \mathcal{Y}$. Then, the function $F(\cdot;\cdot)$ is called $({\mathrm{R}}_{{\mathcal{G}}},{\mathcal B})$-uniformly continuous if, for every $B\in {\mathcal B}$ and for every two sequences $({\bf a}_{k}=(a_{k}^{1},a_{k}^{2},\cdots ,a_{k}^{m}))$, $ ({\bf b}_{k}=(b_{k}^{1},b_{k}^{2},\cdot \cdot\cdot ,b_{k}^{m}))$ $ \in {\mathrm{R}}_{{\mathcal{G}}}$ such that $ ||{\bf a}_{k}-{\bf b}_{k} || \rightarrow 0$, as $k\to +\infty$,  we have 
\begin{align}
\lim_{k\rightarrow +\infty } \sup_{x\in B}\| F \bigl({\bf t}+{\bf a}_{k};x\bigr) - F\bigl({\bf t}+{\bf b}_{k};x\bigr) \|_{\mathcal{Y}} \rightarrow 0,
\end{align}
pointwise for ${\bf t} \in \mathbb{R}^m$.
\end{definition}

\begin{rem}
It is straightforward to verify that Definition \ref{UnifContF} implies Definition \ref{(R,B) uniformly continuous}. However, the converse does not hold in general. To see this, consider
\[
F:\mathbb{R}^2\times\mathbb{R}\to\mathbb{R},
\qquad
F(t,s;x)=k f(t-s),
\]
where $k\in\mathbb{R}$ is a non-zero constant, and $f:\mathbb{R}\to\mathbb{R}$ is a continuous but not uniformly continuous function. Let
\[
\mathcal{G}=\{(a,a):a\in\mathbb{R}\},
\]
that is, the diagonal of $\mathbb{R}^2$, and let $\mathcal{B}$ be, for instance, the family of bounded subsets of $\mathbb{R}$. Then $F$ is $({\mathrm R}_{\mathcal G},\mathcal B)$-uniformly continuous. Indeed, for any ${\bf a}=(a,a),{\bf b}=(b,b)\in{\mathrm R}_{\mathcal G}$ and any ${\bf t}=(t,s)\in\mathbb{R}^2$, we have
\[
F({\bf t}+{\bf a};x)-F({\bf t}+{\bf b};x)
=
k f\bigl((t+a)-(s+a)\bigr)
-k f\bigl((t+b)-(s+b)\bigr)
=0.
\]
Consequently, the condition in Definition \ref{(R,B) uniformly continuous} is satisfied trivially, regardless of the choice of $f$.

On the other hand, $F$ does not satisfy Definition \ref{UnifContF}. Indeed, if $F$ were uniformly continuous in the first variable, then, by taking ${\bf t}=(t,0)$ and ${\bf s}=(s,0)$, we would obtain the uniform continuity of $f$, which contradicts our assumption. Thus, Definition \ref{(R,B) uniformly continuous} does not imply Definition \ref{UnifContF} in general.
\end{rem}

\begin{rem}
If $\mathcal{G}=\mathbb{R}^m$ and ${\mathrm R}_{\mathcal G}={\mathrm R}_{\mathcal G}^{a}$, then Definitions \ref{UnifContF} and \ref{(R,B) uniformly continuous} are equivalent.
\end{rem}

\subsection{Some general results}

We start by showing (in the next Proposition) that, the limit function (in Definition \ref{eovako}) of  a bounded $({\mathrm{R}}_{{\mathcal{G}}},{\mathcal B})$-multi-almost automorphic function, is also bounded.

\begin{prop}\label{ThRKAA boundedness of $F^*$}
Let  $F : {\mathbb R}^{m} \times \mathcal{X} \rightarrow \mathcal{Y}$ be a bounded continuous function. 
If $F$ is $({\mathrm{R}}_{{\mathcal{G}}},{\mathcal B})$-multi-almost automorphic, then its limit function $F^*(\cdot ; \cdot)$ (see Definition \ref{eovako}) is also bounded. Moreover, we have

$$ \sup_{(\t,x) \in \mathbb{R}^m\times \X }|| F^*(\t ; x) ||_{\mathcal{Y}} = \sup_{(\t,x) \in \mathbb{R}^m\times \X}|| F(\t ; x) ||_{\mathcal{Y}}\, .$$

 \end{prop}
 
 \begin{proof}
Let $({\bf b}_{k}) \in {\mathrm{R}}_{{\mathcal{G}}}$ and let $x\in \mathcal{X}$. Then, there exists $B\in \mathcal{B}$ such that $x \in B$ and, according to Definition \ref{eovako}, there exist a subsequence $({\bf b}_{k_{l}})\subset({\bf b}_{k})$ and a function $F^{*}: {\mathbb R}^{m} \times \mathcal{X} \rightarrow \mathcal{Y}$ such that the following limits, pointwise in ${\bf t}$ and uniformly in $x\in B$, hold
$$\lim_{l\rightarrow\infty}{F(\t+{\bf b}_{k_{l}};x)=F^*(\t;x)}\, ,$$
$$\lim_{l\rightarrow\infty}{F^*(\t-{\bf b}_{k_{l}};x)=F(\t;x)}.$$ 
In particular, the previous limits are also pointwise in $x$.
Now, using the first and second (pointwise in $x$) limits, we have the next inequalities

\begin{equation}\label{EeQqNn-01}
||F^*(\t;x)||_{\mathcal{Y}} = \lim_{l\rightarrow\infty}||F(\t+{\bf b}_{k_{l}};x)||_{\mathcal{Y}} \leq \sup_{(\t,x) \in \mathbb{R}^m\times \X}|| F(\t ; x) ||_{\mathcal{Y}} \, ,
\end{equation}
\begin{equation}\label{EeQqNn-02}
||F(\t;x)||_{\mathcal{Y}} = \lim_{l\rightarrow\infty}||F^*(\t-{\bf b}_{k_{l}};x)||_{\mathcal{Y}} \leq \sup_{(\t,x) \in \mathbb{R}^m\times \X}|| F^*(\t ; x) ||_{\mathcal{Y}} \, .
\end{equation}

From (\ref{EeQqNn-01}) we conclude that $F^*$ is bounded and, from (\ref{EeQqNn-01}) and (\ref{EeQqNn-02}), we have

$$ \sup_{(\t,x) \in \mathbb{R}^m\times \X }|| F^*(\t ; x) ||_{\mathcal{Y}} = \sup_{(\t,x) \in \mathbb{R}^m\times \X}|| F(\t ; x) ||_{\mathcal{Y}}\, .$$ 
 
 \end{proof}

\begin{prop}\label{LlEeMmaN01}
Let $F : {\mathbb R}^{m} \times \mathcal{X} \rightarrow \mathcal{Y}$ be continuous in $\mathbb{R}^m$, uniformly for any $B\in {\mathcal B}$. If $F$ is compactly $({\mathrm{R}}_{{\mathcal{G}}},{\mathcal B})$-multi-almost automorphic then, its limit function $F^*$ (see Definition \ref{eovako}) is  continuous in $\mathbb{R}^m$, uniformly for any $B\in {\mathcal B}$.
\end{prop}

\begin{proof}
 Since $F$ is compactly $({\mathrm{R}}_{{\mathcal{G}}},{\mathcal B})$-multi-almost automorphic, given a sequence $({\bf b}_{k}) \in {\mathrm{R}}_{{\mathcal{G}}}$, there exist a subsequence $({\bf b}_{{k}_l}) \subset ({\bf b}_{k}) $ and a function $F^*$ such that the following limits holds

\begin{equation}\label{NEWEQN01}
\lim_{l\to \infty}\sup_{\t \in K} \sup_{x\in B}\| F \bigl({\bf t}+b_{k_l};x\bigr) - F^*\bigl({\bf t};x\bigr) \|_{\mathcal{Y}}=0 \, ,
\end{equation}

$$
\lim_{l\to \infty}\sup_{\t \in K} \sup_{x\in B}\| F^* \bigl({\bf t}-b_{k_l};x\bigr) - F\bigl({\bf t};x\bigr) \|_{\mathcal{Y}}=0 \, ,
$$
where $K$ is any compact subset of $\mathbb{R}^m$ and $B $ is any element in $\mathcal{B}$.



Let $t_0 \in \mathbb{R}^m$. Let us consider the compact set $K_0=B[\t_0,1]$ (the closed ball of center $\t_0$ and radius $1$). Then, given $\epsilon >0$ and using $K=K_0$ in (\ref{NEWEQN01}), there exists $l_0 \in \mathbb{N}$ such that 
(for any $\t \in K_0$)
\begin{eqnarray*}
||F^*(\t;x) - F^*(\t_0,x)||_{\mathcal{Y}} & \leq &  \| F \bigl({\bf t}+b_{k_{l_0}};x\bigr) - F^*\bigl({\bf t};x\bigr) \|_{\mathcal{Y}} + \| F \bigl({\bf t}_0 + b_{k_{l_0}};x\bigr) - F^*\bigl({\bf t}_0;x\bigr) \|_{\mathcal{Y}} +\\
& + & || F \bigl({\bf t} + b_{k_{l_0}};x\bigr)  - F\bigl({\bf t}_0 + b_{k_{l_0}};x\bigr)  ||_{\mathcal{Y}}\\
&\leq & 2 \sup_{\t \in K_0} \sup_{x\in B} \| F \bigl({\bf t}+b_{k_{l_0}};x\bigr) - F^*\bigl({\bf t};x\bigr) \|_{\mathcal{Y}} +\\
&+&  || F \bigl({\bf t} + b_{k_{l_0}};x\bigr)  - F\bigl({\bf t}_0 + b_{k_{l_0}};x\bigr)  ||_{\mathcal{Y}}\\
 & \leq & 2\epsilon/3 + \sup_{x\in B} || F \bigl({\bf t} + b_{k_{l_0}};x\bigr)  - F\bigl({\bf t}_0 + b_{k_{l_0}};x\bigr)  ||_{\mathcal{Y}}\, .
\end{eqnarray*}
From the continuity of $F$ in $ {\bf t}_0 + b_{k_{l_0}}$, uniformly for any $B \in \mathcal{B}$, there exists $\delta_0 = \delta_0(\epsilon,{\bf t}_0 + b_{k_{l_0}},B)>0$ such that if $||\xi -\t_0|| < \delta_0$, then

$$
\sup_{x\in B}\| F \bigl(\xi;x\bigr) - F\bigl({\bf t}_0 + b_{k_{l_0}};x\bigr) \|_{\mathcal{Y}}<\epsilon/3\, .
$$
Now, let us consider $0< \delta  \leq \min\{1, \delta_0 \}$, then if $||\t - \t_0||<\delta $ (and from the previous inequalities) we have
$$\sup_{x\in B} ||F^*(\t;x) - F^*(\t_0,x)||_{\mathcal{Y}} <\epsilon\, .$$
Therefore, the continuity of $F^*$ in $\t_0$, uniformly for any $B\in \mathcal{B}$, follows. Since $\t_0 \in \mathbb{R}^m$ was arbitrary, the lemma holds.

\end{proof}


The following result essentially states that the uniform continuity of $F$ with respect to its first variable is inherited by its limit function $F^*$.

\begin{prop}\label{prou}
    Let $F:\mathbb{R}^m \times\X \to  \mathcal{Y}$ be $({\mathrm{R}}_{{\mathcal{G}}},{\mathcal B})$-multi-almost automorphic and  uniformly continuous in the first variable; then, its limit function $F^*$(see definition \ref{eovako}) is also uniformly continuous in the first variable.
\end{prop}
\begin{proof}

Let $({\bf b}_{k}) \in {\mathrm{R}}_{{\mathcal{G}}}$ and $B\in \mathcal{B}$, since $F$ is $({\mathrm{R}}_{{\mathcal{G}}},{\mathcal B})$-multi-almost automorphic there exist a subsequence $({\bf b}_{k_{l}})\subset({\bf b}_{k})$ and a function $F^*:\mathbb{R}^m \times \X \to \mathcal{Y}$ such that the following limits, point-wise in ${\bf t}$ and uniformly in $x\in B$, hold
$$\lim_{l\rightarrow\infty}{F(\t+{\bf b}_{k_{l}};x)=F^*(\t;x)}\, ,$$
$$\lim_{l\rightarrow\infty}{F^*(\t-{\bf b}_{k_{l}};x)=F(\t;x)}.$$
Now, let $\epsilon > 0$, since $F$ is uniformly continuous in the first variable, there exists $\delta > 0$ such that if $\t,\s \in \mathbb{R}^m$ with $||\t-\s||<\delta$, then $\sup_{x\in B}||F(\t;x) - F(\s ;x)||_{\mathcal{Y}}<\epsilon/2$. Therefore, taking $\sup_{x\in B}$ and then the limit when $l\rightarrow\infty$ in the next inequality
\begin{eqnarray*}
    ||F^*(\t;x) - F^*(\s;x)||_{\mathcal{Y}} &\leq & || F^*(\t;x) - F(\t+{\bf b}_{k_{l}};x)||_{\mathcal{Y}} + ||F(\t+{\bf b}_{k_{l}};x) - F(\s + {\bf b}_{k_{l}};x)||_{\mathcal{Y}} +\\
    &+& ||F(\s + {\bf b}_{k_{l}};x)-F^*(\s ;x)||_{\mathcal{Y}}\, ,
\end{eqnarray*}
we have
$$\sup_{x\in B}||F^*(\t,x)-F^*(\s,x)||_{\mathcal{Y}} < \epsilon\, ,$$
provided that $\t,\s \in \mathbb{R}^m$ with $||\t-\s||<\delta$. Thus, $F^*$ is uniformly continuous in the first variable.
\end{proof}

The proof of the next Theorem is inspired by  \cite[Lemma 3.4]{chavez2024compact}. The result is used, for instance, to establish the uniform continuity of functions defined by integrals with kernels depending on two variables (see, for instance,  Theorems  \ref{NewThmappl} and \ref{InvConv1}). \color{red} The Theorem is also true pointwise in $x \in \mathcal{X}$.
\color{black}

\begin{tw}\label{ThRKAA}
Let  $F : {\mathbb R}^{m} \times \mathcal{X} \rightarrow \mathcal{Y}$ be bounded and continuous. Let $\mathrm{R}_{\mathcal{G}}$ be the set of sequences such that all of its subsequences are also in $\mathrm{R}_{\mathcal{G}}$. If $F$ is compactly $({\mathrm{R}}_{{\mathcal{G}}},{\mathcal B})$-multi-almost automorphic \color{blue} and continuous in $\mathbb{R}^m$, uniformly for any $B\in {\mathcal B}$\color{black}, then $F$ is $({\mathrm{R}}_{{\mathcal{G}}},{\mathcal B})$-multi-almost automorphic and $({\mathrm{R}}_{{\mathcal{G}}},{\mathcal B})$-uniformly continuous.
\end{tw}
\begin{proof}

Since \( F \) is compactly \( ({\mathrm{R}}_{{\mathcal{G}}},{\mathcal{B}}) \)-multi-almost automorphic, it follows that \( F \) is \( ({\mathrm{R}}_{{\mathcal{G}}},{\mathcal{B}}) \)-multi-almost automorphic. Now, let us prove that $F$ is $({\mathrm{R}}_{{\mathcal{G}}},{\mathcal B})$-uniformly continuous.

Let \( B \) be any set in \( \mathcal{B} \), and let us pick two arbitrary sequences \( ( {\bf a}_{k})\) and \( ({\bf b}_{k}) \in {\mathrm{R}}_{{\mathcal{G}}} \) which satisfy \( ||{\bf a}_{k} - {\bf b}_{k} || \rightarrow 0 \) as \( k \to \infty \). We must show that, for each \( {\bf t} \in \mathbb{R}^m \):
\begin{align*}
\lim_{k\to \infty}\sup_{x\in B}\| F \bigl({\bf t}+{\bf a}_{k};x\bigr) - F\bigl({\bf t}+{\bf b}_{k};x\bigr) \|_{\mathcal{Y}} = 0\, .
\end{align*}
Let ${\bf t}\in \mathbb{R}^m$ be fixed, define the uniformly bounded sequence 
$$ \alpha_{k}({\bf t}):= \sup_{x\in B}\| F \bigl({\bf t}+{\bf a}_{k};x\bigr) - F\bigl({\bf t}+{\bf b}_{k};x\bigr) \|_{\mathcal{Y}}\, .$$
 We claim that $\alpha_{k}({\bf t}) \to 0$ as $k\to \infty$. In fact, let $( \alpha_k'({\bf t})) \subset (\alpha_k({\bf t}))$ be a convergent subsequence which converges to $\psi_0({\bf t})$, where
$$ \alpha_k'({\bf t})= \sup_{x\in B}\| F \bigl({\bf t}+{\bf a}_{k}^{'};x\bigr) - F\bigl({\bf t}+{\bf b}_{k}^{'};x\bigr) \|_{\mathcal{Y}}$$
and  $ ({\bf a}_{k}^{'}) \subset ({\bf a}_{k})$, $({\bf b}_{k}^{'}) \subset ({\bf b}_{k})$ are subsequences coming from the definition of $ \alpha_k'({\bf t})$. 

Since $F$ is compactly $({\mathrm{R}}_{{\mathcal{G}}},{\mathcal B})$-multi-almost automorphic, there exist a subsequence $({\bf b}_{k_{l}}=(b_{k_{l}}^{1},b_{k_{l}}^{2},\cdot \cdot\cdot , b_{k_{l}}^{m}))$ of $({\bf b}_{k}^{'})$ (also a subsequence $({\bf a}_{k_{l}})$ of $({\bf a}_{k}^{'})$) and a continuous (in the first variable) function
$F^{\ast} : {\mathbb R}^{m} \times \mathcal{X} \rightarrow \mathcal{Y}$ such that the following limits hold

\begin{align}\label{love12345678}
\lim_{l\rightarrow +\infty} \sup_{{\bf z}\in \mathcal{E}} \sup_{x\in B} \| F\bigl({\bf z} +{\bf b}_{k_{l}};x\bigr)-F^{\ast}({\bf z};x) \|_{\mathcal{Y}}=0\, ,
\end{align}
and
\begin{align}\label{love123456789}
\lim_{l\rightarrow +\infty} \sup_{{\bf z}\in \mathcal{E}} \sup_{x\in B} \| F^{\ast}\bigl({\bf z} -{\bf b}_{k_{l}};x\bigr)-F({\bf z};x)\|_{\mathcal{Y}}=0\, ;
\end{align}
where, $\mathcal{E}$ is any compact subset of $\mathbb{R}^m$.

On the other hand, because  $(||{\bf a}_{k}-{\bf b}_{k}||)$ is bounded, there exists $R>0$ such that for all $l\in \mathbb{
N}$, ${\bf c}_{k_{l}}:={\bf a}_{k_{l}}-{\bf b}_{k_{l}} \in  B[0,R]$ (the closed ball with center $0$ and radius $R$). Therefore, we have:
\begin{align*}
\| F \bigl({\bf t}+{\bf a}_{k_{l}};x\bigr) &- F\bigl({\bf t}+{\bf b}_{k_{l}};x\bigr) \|_{\mathcal{Y}} \leq \\
&\leq  \| F \bigl({\bf t}+{\bf c}_{k_{l}}+{\bf b}_{k_{l}};x\bigr) - F^{\ast}\bigl({\bf t}+{\bf c}_{k_{l}};x\bigr) \|_{\mathcal{Y}}  +\\ 
&+ \|F \bigl({\bf t}+{\bf b}_{k_{l}};x\bigr)- F^{\ast}\bigl({\bf t};x\bigr)  \|_{\mathcal{Y}} +\|F^{\ast} \bigl({\bf t};x\bigr)- F^{\ast}\bigl({\bf t}+{\bf c}_{k_{l}} ;x\bigr)  \|_{\mathcal{Y}} \\
&\leq \sup_{{\bf z}\in \overline{ B(0,R)}+{\bf t}}\sup_{x \in B}\| F \bigl({\bf z}+{\bf b}_{k_{l}};x\bigr) - F^{\ast}\bigl({\bf z};x\bigr) \|_{\mathcal{Y}} +\\ 
&+ \|F \bigl({\bf t}+{\bf b}_{k_{l}};x\bigr)- F^{\ast}\bigl({\bf t};x\bigr)  \|_{\mathcal{Y}} +\|F^{\ast} \bigl({\bf t};x\bigr)- F^{\ast}\bigl({\bf t}+{\bf c}_{k_{l}} ;x\bigr)  \|_{\mathcal{Y}}  .
\end{align*}
 
Now, since $F(\cdot,\cdot)$ is compactly $({\mathrm{R}}_{{\mathcal{G}}},{\mathcal B})$-multi-almost automorphy and  $F^{\ast} \bigl(\cdot;x\bigr)$ is continuous in $\mathbb{R}^m$, uniformly for any $B\in {\mathcal B}$ (see Proposition \ref{LlEeMmaN01}), we have (using the previous inequality), that
$$\lim_{l \to \infty}\sup_{x\in B} \| F \bigl({\bf t}+{\bf a}_{k_{l}};x\bigr) - F\bigl({\bf t}+{\bf b}_{k_{l}};x\bigr) \|_{\mathcal{Y}}  = 0\, .$$
Therefore, $\psi_0({\bf t})=0$. Now, since every convergent subsequence of the bounded sequence $\alpha_{k}({\bf t})$ converge towards zero, we conclude 
$$\lim_{k \to \infty} \sup_{x\in B} \| F \bigl({\bf t}+{\bf a}_{k};x\bigr) - F\bigl({\bf t}+{\bf b}_{k};x\bigr) \|_{\mathcal{Y}} =0\, .$$

\end{proof}

The next theorem provides a partial converse to Theorem \ref{ThRKAA}.

\begin{tw}\label{ThRKAA Partial converse}
Let  $F : {\mathbb R}^{m} \times \mathcal{X} \rightarrow \mathcal{Y}$ be bounded and continuous. 
If $F$ is $({\mathrm{R}}_{{\mathcal{G}}},{\mathcal B})$-multi-almost automorphic and uniformly continuous in the first variable, then $F$ is compactly $({\mathrm{R}}_{{\mathcal{G}}},{\mathcal B})$-multi-almost automorphic and continuous in $\mathbb{R}^m$, uniformly for any $B\in {\mathcal B}$.
 \end{tw}
\begin{proof}
As a first instance note that the continuity of $F$ in $\mathbb{R}^m$, uniformly for any $B\in {\mathcal B}$, is a consequence of the uniform continuity of $F$ in the first variable. Let us prove that $F$ is compactly $({\mathrm{R}}_{{\mathcal{G}}},{\mathcal B})$-multi-almost automorphic.

Let $({\bf b}_{k}) \in {\mathrm{R}}_{{\mathcal{G}}}$ and $B\in \mathcal{B}$, since $F$ is $({\mathrm{R}}_{{\mathcal{G}}},{\mathcal B})$-multi-almost automorphic there exist a subsequence $({\bf b}_{k_{l}})\subset({\bf b}_{k})$ and a function $F^*:\mathbb{R}^m \times \X \to \mathcal{Y}$ such that the following limits hold
\begin{align}\label{love123456781}
\lim_{l\rightarrow +\infty}\sup_{x\in B}|| F\bigl({\bf t} +{\bf b}_{k_{l}};x\bigr)-F^{\ast}({\bf t};x) ||_{\mathcal{Y}}=0\, ,
\end{align}
and
\begin{align}\label{love1234567891}
\lim_{l\rightarrow +\infty}\sup_{x\in B}|| F^{\ast}\bigl({\bf t} -{\bf b}_{k_{l}};x\bigr)-F({\bf t};x)||_{\mathcal{Y}}=0\, ,
\end{align}
pointwise for ${\bf t}\in {\mathbb R}^{m}$.

Since F is bounded and uniformly continuous with respect to the first variable, it follows from Proposition \ref{ThRKAA boundedness of $F^*$} and Proposition \ref{prou} that $F^*$ 
 is also bounded and uniformly continuous in the first variable.
 
For each $l\in \mathbb{N}$, define  $F_l: \mathbb{R}^m \to \mathbb{R}$ by
$$F_l(\t):=\sup_{x\in B}|| F\bigl({\bf t} +{\bf b}_{k_{l}};x\bigr)-F^{\ast}({\bf t};x) ||_{\mathcal{Y}}\, .$$ 
Using the elementary inequality
\begin{equation}\label{Elementary inequality}
|\, ||a||_{\mathbb{A}}-||b||_{\mathbb{A}}\, |\leq ||a-b||_{\mathbb{A}}, \, \, \forall a,b \in \mathbb{A}\, ;
\end{equation}
where $\mathbb{A}$ is a normed space with norm $||\cdot||_{\mathbb{A}}$, one can show that, for each $l\in \mathbb{N}$,  $F_l \in C(\mathbb{R}^m ; \mathbb{R})$.

 Let $K\subset\mathbb{R}^{m}$ be a compact set and define the family $\mathcal{F}:=\lbrace F_{l}: l \in\mathbb{N}\rbrace \subset C(K; \mathbb{R})$. We note that for each $\t \in K$ the set $\{ F_l(\t)\, :\, l \in \mathbb{N}\}$ is relatively compact in $\mathbb{R}$. Let us prove that the family $\mathcal{F}$ is equicontinuous.

Let $\epsilon >0$, from the uniform continuity in the first variable of $F$ and $F^*$, there exists $\delta = \delta(\epsilon)>0$ such that if $\t, {\bf s} \in \mathbb{R}^m$ with $||\t - {\bf s}||<\delta$, then
\begin{equation}\label{EeEeQqQqNnNn01}
\sup_{x\in B}|| F\bigl({\bf t} ;x\bigr)-F({\bf s};x) ||_{\mathcal{Y}} < \epsilon /3 \, ,
\end{equation}
and
\begin{equation}\label{EeEeQqQqNnNn02}
\sup_{x\in B}|| F^*\bigl({\bf t};x\bigr)-F^{\ast}({\bf s};x) ||_{\mathcal{Y}} < \epsilon /3 \, .
\end{equation}
Therefore, using the elementary inequality (\ref{Elementary inequality}), jointly with (\ref{EeEeQqQqNnNn01}) and (\ref{EeEeQqQqNnNn02}), we deduce that
%
%
\begin{eqnarray*}
| F_l(\t)- F_l({\bf s}) | <  \epsilon\, , \, \, \, \forall l \in \mathbb{N}\, ,
\end{eqnarray*}
and for every $\t, {\bf s} \in \mathbb{R}^m$ with $||\t - {\bf s}||<\delta$. 
Thus, $\mathcal{F}$ is an equicontinuous Family. 

Applying the Arzelà-Ascoli Theorem, there exists a subsequence $({\bf a}_{k_l})$ of $({\bf b}_{k_{l}})$ such that the following limit hold
\begin{align}\label{love1234567811}
\lim_{l\rightarrow +\infty}\sup_{\t \in K}\sup_{x\in B}|| F\bigl({\bf t} +{\bf a}_{k_{l}};x\bigr)-F^{\ast}({\bf t};x) ||_{\mathcal{Y}}=0\, .
\end{align}

Now, for each $j \in \mathbb{N}$, let $K_j=B[0,j]$ (the closed ball of center $0$ and radius $j$), then $K_j$ is compact, $K_j \subset K_{j+1}$ and $\cup_{j \in \mathbb{N}}K_j =\mathbb{R}^m$. Therefore, applying the previous argument in each compact set $K_j$ we can identify a (diagonal)  subsequence $({\bf c}_{k_l})$ of $({\bf b}_{k})$  such that, on any compact set $K \subset \mathbb{R}^m$, the following limit holds 
\begin{align}\label{LLlove123456781101}
\lim_{l\rightarrow +\infty}\sup_{\t \in K}\sup_{x\in B}|| F\bigl({\bf t} +{\bf c}_{k_{l}};x\bigr)-F^{\ast}({\bf t};x) ||_{\mathcal{Y}}=0\, .
\end{align}
Following the same reasoning, there exists a subsequence $({\bf s}_{k_l})$ of $({\bf c}_{k_l})$ such that, on any compact set $K\subset
\mathbb{R}^m$, the following limit holds
\begin{align}\label{LLlove1234567891101}
\lim_{l\rightarrow +\infty}\sup_{\t \in K}\sup_{x\in B}|| F^{\ast}\bigl({\bf t} -{\bf s}_{k_{l}};x\bigr)-F({\bf t};x)||_{\mathcal{Y}}=0\, .
\end{align}
Therefore, there exists a subsequence $({\bf s}_{k_l})$ of $({\bf b}_{k_l})$ such that the Limits in (\ref{LLlove123456781101}) and (\ref{LLlove1234567891101}) hold.
Thus, $F$ is compactly $({\mathrm{R}}_{{\mathcal{G}}},{\mathcal B})$-multi-almost automorphic.

\end{proof}

The proof of the next proposition, which is omitted, is analogous to that of Proposition \ref{prou}
\begin{prop}
Suppose that \( F : {\mathbb R}^{m} \times \mathcal{X} \rightarrow \mathcal{Y} \) is an \(({\mathrm{R}}_{{\mathcal{G}}}, {\mathcal{B}})\)-multi-almost automorphic function with limit function \( F^* \) (see Definition \ref{eovako}). If for every \( B \in {\mathcal{B}} \) there exists a real number \( L_{B} > 0 \) such that, for every \( x \in B \) and \( {\bf s}, {\bf t} \in {\mathbb R}^{m} \), we have
\[
\bigl\| F({\bf s}; x) - F({\bf t}; x) \bigr\|_{\mathcal{Y}} \leq L_{B} \|{\bf s} - {\bf t}\|,
\]
then, we also have 
\[
\bigl\| F^*({\bf s}; x) - F^*({\bf t}; x) \bigr\|_{\mathcal{Y}} \leq L_{B} \|{\bf s} - {\bf t}\| , 
\]
for every $x\in B$ and ${\bf s}, {\bf t} \in \mathbb{R}^m$.
\end{prop}

\noindent Following the work \cite{chavez2022multi}, we have:
\begin{tw}\label{Follprev}
The following properties hold
\begin{enumerate}
\item Let $\mathrm{R}_{\mathcal{G}}$ be the set of sequences such that all of its subsequences are also in $\mathrm{R}_{\mathcal{G}}$. The space of bounded and ${\mathrm{R}}_{{\mathcal{G}}}$-multi-almost automorphic functions is a Banach space under the supremum norm.

\item  Suppose that $F : {\mathbb R}^{m} \times \mathcal{X} \rightarrow \mathcal{Y}$ is $(\mathcal{G},{\mathcal B})$-multi-almost automorphic, where 
 ${\mathcal B}$ denotes any collection of compact subsets of $\mathcal{X}.$ If for every $B\in {\mathcal B}$ there exists a real number $L_{B}>0$ such that, for every $x,\ y\in B$ and ${\bf t}\in {\mathbb R}^{m},$
we have
\begin{align}\label{elbe}
\bigl\| F({\bf t};x) -F({\bf t} ;y)\bigr\|_{\mathcal{Y}}\leq L_{B}\|x-y\|_{\mathcal{X}} ;
\end{align} 
then, for every set $B\in {\mathcal B},$ the set $\{F({\bf t},x) : {\bf t}\in \mathcal{G},\ x\in B\}$ is relatively compact in $\mathcal{Y}.$
\end{enumerate}
\end{tw}
In Theorem \ref{Follprev}, the proof of part {\it 1} is analogous to the classical version of almost automorphy, see for instance \cite{CHAKPPINTO2023}, while the proof of item {\it 2} is analogous to \cite[Proposition 2.5-(i)]{chavez2022multi}.

With respect to the previous Proposition, we clarify the following: in item {\it 1}, the completeness is of the space \( BC(\mathbb{R}^m; \mathcal{X}) \cap \mathrm{R}_{\mathcal{G}} A A(\mathbb{R}^m; \mathcal{X}) \) and not just of the  space \( \mathrm{R}_{\mathcal{G}} A A(\mathbb{R}^m; \mathcal{X})  \), this is due to the fact that, in general, the elements of \( \mathrm{R}_{\mathcal{G}} A A(\mathbb{R}^m; \mathcal{X}) \) are not bounded (see the examples presented in  \cite{chavez2022multi}). With respect to item {\it 2}, we clarify that we are working with all the sequences in \( \mathcal{G}\), so that when \( \mathcal{G} = \mathbb{R}^m \), we are dealing with the classical situation of almost automorphy, as in \cite[Proposition 2.5-(i)]{chavez2022multi}. In the next section we prove that if \( \mathcal{G} \) is a dense subgroup of \( \mathbb{R}^m \), then every function in \( \mathcal{G} A A(\mathbb{R}^m; \mathcal{X}) \) is bounded  and has relatively compact range. 

\section{$\mathcal{G}$-multi-almost automorphic functions, where $\mathcal{G}$ is a dense subgroup of $\mathbb{R}^m$}\label{section dense} 

As quoted in the previous section, we know that an $({\mathrm R}_{\mathcal{G}},{\mathcal B})$-multi-almost automorphic function is not necessarily bounded and, therefore, does not have a relatively compact range. In the next theorem, we overcome this situation by requiring $\mathcal{G}$ to be a dense subgroup of $\mathbb{R}^m$.
\begin{tw}\label{teo1}
Let $\mathcal{G} $ be a dense subgroup of $\mathbb{R}^m$ and $F \in \mathcal{G}AA(\mathbb{R}^m ; \mathcal{X})$, then:
\begin{enumerate}
\item $F$ is bounded.
\item $F$ has relatively compact range. Thus, $\mathcal{R}(F):=\{F({\bf t}) : {\bf t}\in \mathbb{R}^m\}$ is relatively compact in $\mathcal{X}$.
\end{enumerate}
\end{tw}
\begin{proof}
Let us see:
\begin{enumerate}
\item [{\it 1.}] Let $F_r := F|_{\mathcal{G}} : \mathcal{G} \rightarrow \X$ be the restriction of $F$ to $\mathcal{G}$, which is a continuous function. We note that $F_r$ is almost automorphic on $\mathcal{G}$ and, according to Theorem \ref{pro1}, $F_r$ is bounded on $\mathcal{G}$. Consequently, due to the continuity of $F$, the boundedness of $F_r$ and the density of $\mathcal{G}$ in $\mathbb{R}^m$, it follows that $F$ is bounded on $\mathbb{R}^m$.

\item [{\it 2.}] We must prove that $\overline{F(\mathbb{R}^m)}$ is compact in $\mathcal{X}$. Since $\mathcal{G}$ is dense in $\mathbb{R}^m$, we have $\overline{F(\mathbb{R}^m)} = \overline{F(\overline{\mathcal{G}})}$ and, according to Proposition \ref{grp}, we have  $\overline{F(\overline{\mathcal{G}})} = \overline{F_r(\mathcal{G})}$. That is, the following equalities hold
\begin{equation}\label{Eqnclos}
\overline{F(\mathbb{R}^m)} = \overline{F(\overline{\mathcal{G}})} = \overline{F_r(\mathcal{G})}\, ;
\end{equation}
where, $F_r$ is the restriction of $F$ to the group $\mathcal{G}$. Moreover, since $F_r: \mathcal{G} \rightarrow \X$ is almost automorphic then, by Theorem \ref{pro2}, $F_r$ has relatively compact range, that is $\overline{F_r(\mathcal{G})}$ is compact in $\mathcal{X}$. Therefore, from (\ref{Eqnclos}), we conclude that $\overline{F(\mathbb{R}^m)}$ is compact in $\mathcal{X}$.
\end{enumerate}
\end{proof}

\begin{cor} Let $\mathcal{G}$ be a dense subgroup of $\mathbb{R}^m$, then
\begin{enumerate}
\item $\mathcal{G}AA(\mathbb{R}^m ; \mathcal{X}) \subset BC(\mathbb{R}^m ; \mathcal{X}).$ 
\item $\mathcal{G}AA(\mathbb{R}^m ; \X)$ is a Banach space under the supremum norm.
\end{enumerate}
\end{cor}

Examining the proof of Theorem \ref{teo1}, we can derive the following result, whose proof is exactly the same.

\begin{tw}
Let $\mathcal{G}$ be a dense subgroup of $\mathbb{R}^m$, then we have:
\begin{enumerate}
\item $AA(\mathcal{G}; \X) \cap C(\mathbb{R}^m; \X) \subset BC(\mathbb{R}^m; \X)$.
\item If $F\in AA(\mathcal{G}; \X) \cap C(\mathbb{R}^m; \X)$, then $F$ has relatively compact range; thus, $ \mathcal{R}(F):=\{F({\bf t}) : {\bf t}\in \mathbb{R}^m\}$ is relatively compact in $\mathcal{X}$.
\end{enumerate}
\end{tw}

\begin{rem} When $\mathcal{G}$ is a subgroup of $\mathbb{R}^m$, we interpret an element \( F \in AA(\mathcal{G}; \X) \cap C(\mathbb{R}^m; \X) \) as a continuous function \( F: \mathbb{R}^m \to \X \) whose restriction to \( \mathcal{G} \),  \( F_r := F|_{\mathcal{G}} : \mathcal{G} \to \X \), is an almost automorphic function in the sense of Bochner.
\end{rem}

In the following result, we establish that if $\mathcal{G}$ is a dense subgroup of $\mathbb{R}^m$, then any $\mathcal{G}$-multi-almost automorphic function which decays asymptotically across $\mathcal{G}$ must be the null function.
\begin{tw}\label{teo6}
Let $\mathcal{G}$ be a dense subgroup of $\mathbb{R}^m$ and $ F \in \mathcal{G}AA(\mathbb{R}^m ; \X)$ such that 
$$\displaystyle\lim_{\t \in \mathcal{G},\, ||\t||\rightarrow\infty}{||F(\t)||_{\X}} = 0\, ;$$ 
then, $F \equiv 0$. 
\end{tw}
\begin{proof}
Since $F\in \mathcal{G} AA(\mathbb{R}^m ;\mathcal{X})$, then its restriction map $F_r:\mathcal{G}\to \X$ is almost automorphic on the group $\mathcal{G}$. On the other hand, according to Lemma \ref{lem8}, $\mathcal{G}$ is unbounded. Therefore, there exist a sequence $({\bf b}_k)_{k\in\mathbb{N}}$ in $\mathcal{G}$ which is unbounded. Now, since $F_r$ is almost automorphic, there exit a subsequence $({\bf b}_{k_l})_{l\in\mathbb{N}}\subset ({\bf b}_k)_{k\in\mathbb{N}}$ and a function $F_r^*:\mathcal{G} \to \mathcal{X}$ such that for each $\t \in \mathcal{G}$, the following limits hold:
\begin{equation}\label{eq1}
    \lim_{l \rightarrow \infty}{F_r(\t+{\bf b}_{k_l})}=F_r^*(\t) \, ,
\end{equation}
\begin{equation}\label{eq2}
    \lim_{l \rightarrow \infty}{F_r^*(\t-{\bf b}_{k_l})} = F_r(\t)\, .
\end{equation}
But, by hypothesis $\displaystyle\lim_{\t \in \mathcal{G},\, ||\t||\rightarrow\infty}{||F(\t)||_{\X}} = 0;$ then, using (\ref{eq1}), we have:
$$F_r^*(\t) = 0\, ,\, \, \t\in \mathcal{G}\, .$$
Now, using (\ref{eq2}), we also have:
$$F_r(\t)=\lim_{l \rightarrow \infty}{F_r^*(\t-{\bf b}_{k_l})}=0\, ,  \t\in \mathcal{G}\, . $$
Thus, $F_r \equiv 0$ in $\mathcal{G}$. Therefore, from the continuity of $F$ and the denseness of the subgroup $\mathcal{G}$ in $\mathbb{R}^m$, the result follows.. 
\end{proof}

The proof of the next result, follows the same arguments presented in the proof of Theorem \ref{teo6}.

\begin{tw}\label{teo7}
Let $\mathcal{G}$ be a dense subgroup of $\mathbb{R}^m$ and $F\in AA(\mathcal{G}; \X) \cap C(\mathbb{R}^m; \X)$ such that 
$$\displaystyle\lim_{\t \in \mathcal{G},\, ||\t||\rightarrow\infty}{||F(\t)||_{\X}} = 0\, ;$$ 
then, $F \equiv 0$. 
\end{tw}

Although in this section we mainly present results for the 'actions' of dense subgroups of Euclidean space, we do not want to miss the opportunity to complement the last two results with the following theorem.

\begin{tw}
Let $F \in \mathbb{Z}^m AA(\mathbb{R}^m; \X)$ and  $\alpha \in \mathbb{R}$ such that $F(\t)=0$ for all $\t \in N{\alpha}$, where
$$N_{\alpha}:=\left\lbrace \t=(t_1,t_2, \cdots , t_m)\in \mathbb{R}^m\, : \, t_m>\alpha  \right\rbrace\, .$$
Then, $F \equiv 0$.
\end{tw}

\begin{proof}
Let us consider the sequence $({\bf b}_k)_{k\in\mathbb{N}}$ in $\mathbb{Z}^m$ defined, for each $ k \in \mathbb{N}$, by $({\bf b}_{k}=(0,0, \cdots, 0, k))$. Since $F\in \mathbb{Z}^m AA(\mathbb{R}^m; \X)$, there exist a subsequence $({\bf b}_{k_l})_{l\in\mathbb{N}}\subset ({\bf b}_k)_{k\in\mathbb{N}}$ and a function $F^*:\mathbb{R}^m \to \mathcal{X}$ such that for each $\t \in \mathbb{R}^m$, the following limits hold:
\begin{equation}\label{eq10}
    F^*(\t) = \lim_{l \rightarrow \infty}{F(\t+{\bf b}_{k_l})}\, ,
\end{equation}
\begin{equation}\label{eq20}
    F(\t)=\lim_{l \rightarrow \infty}{F^*(\t-{\bf b}_{k_l})} \, .
\end{equation}
Let \( \t \in N_{\alpha}^{c} := \mathbb{R}^m \setminus N_{\alpha} \). Then, for \( l \) large enough, we have \( \t + {\bf b}_{k_l} \in N_{\alpha} \); therefore, from (\ref{eq10}), we have \( F^*(\t) = 0 \). From this conclusion, and with the help of (\ref{eq20}), we conclude that \( F(\t) = 0 \) for every \( \t \in N_{\alpha}^{c} \)\, .
\end{proof}

Certainly, the last Theorems are relevant to the ongoing study (in the spirit of this paper) of the class of asymptotically almost automorphic functions. We intend to explore this class of function space in a future work.

\section{New characterizations of compact almost automorphic functions}\label{section char}

In this section, we present the main results of the present work. We begin with the following well-known characterization of multidimensional compact almost automorphic functions presented in \cite{CHAKPPINTO2023}.
\begin{tw}\label{crt} ({\bf First characterization}) A function $F:\mathbb{R}^{m}\to \mathcal{X}$ is compact almost automorphic if and only if it is almost automorphic and uniformly continuous. Thus,
$$\mathcal{K}AA(\mathbb{R}^m ; \X) = AA(\mathbb{R}^m ; \X) \cap UC(\mathbb{R}^m ; \X) .$$
\end{tw}

Thus, Theorem \ref{crt} establishes that, in order to recognize that a function  \( F: \mathbb{R}^m \to \X \) is compact almost automorphic, it is sufficient to verify that \( F \) is almost automorphic and uniformly continuous. 
The purpose of this section is to consider, instead of the class \( AA(\mathbb{R}^m; \X) \) used in the characterization presented in Theorem \ref{crt}, the broader class of \( \mathcal{G} \)-multi-almost automorphic functions, where \( \mathcal{G} \) is an appropriate subset of \( \mathbb{R}^m \).
 
In the first new characterization presented here (the next Theorem), we consider \( \mathcal{G} = \mathbb{Z}^m \). This result extends \cite[Theorem 3.16]{Alanpiecewise} to the multidimensional setting. For a better presentation of the result, we use the notation $\lfloor \cdot \rfloor$ for the floor function, which is defined as follows: let $x\in \mathbb{R}$, then 
\[
\lfloor x \rfloor := \max\{m \in \mathbb{Z} \mid m \leq x\}\, .
\]
The fractional part function, denoted by $\{ \cdot\}$, is defined as follows: let $\in \mathbb{R}$, then 
$$\{ x \} := x- \lfloor x \rfloor\, . $$
Note that, for each $x\in \mathbb{R}$, $\lfloor x \rfloor \in \mathbb{Z}$ and $\{x \} \in [0,1)$.


\begin{tw} ({\bf Second characterization}) 
    A function $F : \mathbb{R}^m \rightarrow \X$ is compact almost automorphic if and only if it is $\mathbb{Z}^m$-almost automorphic and uniformly continuous, that is:
    $$\mathcal{K}AA(\mathbb{R}^m ; \X) = \mathbb{Z}^mAA(\mathbb{R}^m ; \X) \cap UC(\mathbb{R}^m ; \X) .$$
\end{tw}
\begin{proof} 
Let us see\\

\noindent    i) Let $F$ be a compact almost automorphic function, then $F$ is $\mathbb{Z}^m$-multi-almost automorphic and, according to Theorem \ref{crt}, it is uniformly continuous. Therefore we have the inclusion $\mathcal{K}AA(\mathbb{R}^m ; \X)\subset  \mathbb{Z}^mAA(\mathbb{R}^m ; \X) \cap UC(\mathbb{R}^m ; \X).$ \\
\noindent   ii) Let us prove the reverse inclusion. Consider $F$ a $\mathbb{Z}^m$-multi-almost automorphic and uniformly continuous function; then, according to Theorem \ref{crt}, it is enough to prove that $F$ is almost automorphic.\\

Let $({\bf b}_{k}=(b_{k}^{1},b_{k}^{2},\cdots ,b_{k}^{m}))_{k\in \mathbb{N}}$ be an arbitrary sequence in $\mathbb{R}^m$. For each $j\in \{1,2,\cdots , m\}$, we can write the decomposition
\[
b_k^j=\lfloor b_k^j \rfloor+\{b_k^j\}.
\]
Accordingly, we obtain
\begin{eqnarray*}
    {\bf b}_k     &=& \underbrace{(\lfloor b_k^1 \rfloor,\lfloor b_k^2 \rfloor,\cdots,\lfloor b_k^m \rfloor)}_{\in \mathbb{Z}^m}
    + \underbrace{(\{b_k^1\},\{b_k^2\},\cdots,\{b_k^m\})}_{\in [0,1)^m}.
\end{eqnarray*}
Now, since $F \in \mathbb{Z}^mAA(\mathbb{R}^m,\X)$, there exist a subsequence $(\alpha_k)_{k \in\mathbb{N}}$ of $(\lfloor b_k^1 \rfloor,\lfloor b_k^2 \rfloor,\cdots,\lfloor b_k^m \rfloor)_{k \in\mathbb{N}}$ and a function $F^*:\mathbb{R}^m \rightarrow \X$ such that the following pointwise limits hold
\begin{eqnarray}\label{EQq00one}
\lim_{k\rightarrow\infty}{F(\t+\alpha_k)=F^*(\t)}\, , \t \in \mathbb{R}^m\, ,
\end{eqnarray}
\begin{eqnarray}\label{EQq00two}
\lim_{k\rightarrow\infty}{F^*(\t-\alpha_k)=F(\t)}\, , \t \in \mathbb{R}^m\, .
\end{eqnarray}

On the other hand, since for each $k \in \mathbb{N}$ we have $(\{b_k^1\},\{b_k^2\},\cdots,\{b_k^m\}) \in [0,1)^m$, it follows that 
 there exists a convergent subsequence $(\beta_k=(\beta_k^1,\beta_k^2,\cdots,\beta_k^m))_{k\in \mathbb{N}}$ of the sequence  $((\{b_k^1\},\{b_k^2\},\cdots,\{b_k^m\}))_{k\in \mathbb{N}}$, which converges to a point $\beta_0=(\beta_0^1,\beta_0^2,\cdots,\beta_0^m) \in [0,1]^m$. 

In this way, it is possible to obtain a subsequence $(\eta_k)_{k \in \mathbb{N}} \subset ({\bf b}_k)_{k\in \mathbb{N}}$, defined by
$$\eta_k := (\alpha_k^1,\alpha_k^2,\cdots,\alpha_k^m)+(\beta_k^1,\beta_k^2,\cdots,\beta_k^m)=\alpha_k + \beta_k \, ,$$
where $(\alpha_k)_{k}$ satisfies (\ref{EQq00one}) - (\ref{EQq00two}) and $(\beta_k)_{k}$ is such that $\beta_k \to \beta_0$ as $k \to \infty$.

Now, let the function $H:\mathbb{R}^m \to \X$ defined as  $H(\t):= F^*(\t+\beta_0),\, {\bf t} \in \mathbb{R}^m$.

\noindent Let $\t \in \mathbb{R}^m$ be fixed. Then,
\begin{eqnarray}\label{EQNEqna001}
    \|F(\t+\eta_k)-H(\t)\|_{\X}
    &\leq& \|F(\t+\alpha_k+\beta_k)-F(\t+\alpha_k+\beta_0)\|_{\X}  \\ \nonumber
    && + \|F(\t+\alpha_k+\beta_0)-H(\t)\|_{\X}.
\end{eqnarray}
Since $\beta_k \to \beta_0$ as $k \to \infty$, and
\[
\|F(\t + \alpha_k + \beta_0) - H(\t)\|_{\X} \to 0 \quad \text{as } k \to \infty,
\]
it follows, by the uniform continuity of $F$ and (\ref{EQNEqna001}), that
\[
\lim_{k\to\infty} F(\t+\eta_k)=H(\t)\, .
\]

On the other hand, by Proposition \ref{prou}, the function \( F^* \) is uniformly continuous, and hence so is $H$. Therefore, by a completely analogous  argument to the one given above, it follows that
\[
\lim_{k\to \infty} H(\t-\eta_k) = F(\t)\, .
\]
\end{proof}

In the second new characterization presented here (the next Theorem), we consider \( \mathcal{G} \) as a dense subset of $\mathbb{R}^m$.

\begin{tw} ({\bf Third characterization})
    Let $F: \mathbb{R}^m\to \mathcal{X}$ be a continuous function and $\mathcal{G}$ a dense subset of $\mathbb{R}^m$. Then, $F$ is compact almost automorphic if and only if it is $\mathcal{G}$-almost automorphic and uniformly continuous. Thus,
    $$\mathcal{K}AA(\mathbb{R}^m ; \X) = \mathcal{G}AA(\mathbb{R}^m ; \X) \cap UC(\mathbb{R}^m ; \X) .$$
\end{tw}
\begin{proof}
The inclusion $\mathcal{K}AA(\mathbb{R}^m ; \X) \subset \mathcal{G}AA(\mathbb{R}^m ; \X) \cap UC(\mathbb{R}^m ; \X)$ follows from Theorem \ref{crt}. Let us prove the reverse inclusion.

Let $F \in \mathcal{G}AA(\mathbb{R}^m ; \X) \cap UC(\mathbb{R}^m ; \X)$, according to Theorem \ref{crt}, it is enough to prove that $F$ is almost automorphic. Let $({\bf b}_k) \subset \mathbb{R}^m$  be an arbitrary sequence and $\epsilon>0$. For each $k\in \mathbb{N}$, let $B^*({\bf b}_k,\epsilon/2^k):=B({\bf b}_k,\epsilon/2^k)\setminus \{{\bf b}_k \}$, where $B({\bf b}_k,\epsilon/2^k)$ is the open ball with center $ {\bf b}_k$ and radius $\epsilon /2^k$. Since $\mathcal{G}$ is dense in $\mathbb{R}^m$, for each $k\in \mathbb{N}$, there exists  a point $\gamma_k \in B^*({\bf b}_k,\epsilon/2^k)\cap \mathcal{G}$. Equivalently, there exists a sequence  $(\gamma_{k})$ in $\mathcal{G}$ such that $\gamma_{k}\not = {\bf b}_k $ and $||\gamma_{k}-{\bf b}_k||<\epsilon/2^k$, which implies that $\lim_{k\to \infty}||\gamma_{k}-{\bf b}_k||=0$.

 Since $F$ is $\mathcal{G}$-almost automorphic, there exist a subsequence $(\gamma_{k_l}) \subset (\gamma_{k})$ and a function $F^*:\mathbb{R}^m \rightarrow \X$ such that the following pointwise limits hold
\begin{equation}\label{EeQqNnWw}
    \lim_{l\rightarrow\infty}{F(\t+\gamma_{k_l})=F^*(\t)}\, , \, \, \,  \lim_{l\rightarrow\infty}{F^*(\t-\gamma_{k_l})=F(\t)}\, , \, \, {\bf t} \in \mathbb{R}^m\, .
\end{equation}
From Proposition \ref{prou}, it follows that $F^*$ is also uniformly continuous. On the other hand, there exists a subsequence $({\bf b}_{k_l}) \subset ({\bf b}_k)$ such that $\lim_{l\to \infty}||\gamma_{k_l}-{\bf b}_{k_l}||=0$. Now, from the following inequalities
\begin{eqnarray*}
    ||F(\t+{\bf b}_{k_l})-F^*(\t)||_{\mathcal{X}}&\leq & ||F(\t+{\bf b}_{k_l})-F(\t+\gamma_{k_l})||_{\mathcal{X}}+\\
    &+& ||F(\t+\gamma_{k_l})-F^*(\t) ||_{\mathcal{X}}\, ,
\end{eqnarray*}
\begin{eqnarray*}
    ||F^*(\t-{\bf b}_{k_l})-F(\t)||_{\mathcal{X}}&\leq & ||F^*(\t-{\bf b}_{k_l})-F^*(\t-\gamma_{k_l})||_{\mathcal{X}}+\\
    &+& ||F^*(\t-\gamma_{k_l})-F(\t) ||_{\mathcal{X}}\, ,
\end{eqnarray*}
the uniform continuity of $F$, of $F^*$ and the limits in (\ref{EeQqNnWw}) we conclude that
$$\lim_{l\rightarrow\infty}{F(\t+ {\bf b}_{k_l})=F^*(\t)}\, ; $$
and,
$$\lim_{l\rightarrow\infty}{F^*(\t-{\bf b}_{k_l})=F(\t)}.$$
Therefore $F$ is almost automorphic.
\end{proof}

\section{Convolution products}\label{section conv}
\color{blue}
In this section, we study invariance of ${\mathrm R}_{\mathcal{G}}$-multi-almost automorphic function spaces under convolution products. For a better presentation of the results, let us introduce the following notation:  for a point  ${\bf t} = (t_1,t_2,\cdots , t_m) \in \mathbb{R}^m$, $\mathcal{I}_{{\bf t}}=(-\infty,t_1]\times (-\infty,t_2]\times \cdots \times (-\infty,t_m]$.
\color{black}
\begin{tw}\label{TI0101} \cite{chavez2022multi}

\begin{enumerate}
\item Let $f:\mathbb{R}^{m}\rightarrow \X$ be a bounded ${\mathrm R}_{\mathcal{G}}$-multi-almost automorphic function. Let 
$(K({\bf t}))_{{\bf t}\in (0,\infty)^{m}}\subseteq \mathcal{L}(\X, \mathcal{Y})$ be a strongly continuous operator family and 
$$\int_{(0,\infty)^{m}}\| K({\bf t})\|_{\mathcal{L}(\X,\mathcal{Y})}\, d {\bf t}\, < +\infty .$$
Define
\begin{equation}\label{rep01}
F({\bf t}):=\int_{\mathcal{I}_{\bf t}}K({\bf t}-\eta)f(\eta)\, d\eta,\quad {\bf t}\in {\mathbb R}^{m}.
\end{equation}

Then, $F(\cdot)$ is a bounded ${\mathrm R}_{\mathcal{G}}$-multi-almost automorphic function. 

\item Let $f:\mathbb{R}^{m}\rightarrow \X$ be a bounded (compactly) ${\mathrm R}_{\mathcal{G}}$-multi-almost automorphic function. Let 
$(K({\bf t}))_{{\bf t}\in \mathbb{R}^{m}}\subseteq \mathcal{L}(\X, \mathcal{Y})$ be a strongly continuous operator family and 
$$\int_{\mathbb{R}^{m}}\| K({\bf t})\|_{\mathcal{L}(\X,\mathcal{Y})}\, d {\bf t}\, < +\infty .$$
Define
\begin{equation}\label{rep02}
F({\bf t}):=\int_{\mathbb{R}^m}K({\bf t}-\eta)f(\eta)\, d\eta,\quad {\bf t}\in {\mathbb R}^{m}.
\end{equation}
Then, $F(\cdot)$ is a bounded (compactly) ${\mathrm R}_{\mathcal{G}}$-multi-almost automorphic function. 
\end{enumerate}
\end{tw}
\begin{proof}
Let us prove the second item under the hypothesis that $f$ is bounded and compactly ${\mathrm R}_{\mathcal{G}}$-multi-almost automorphic. First observe that since $f$ is bounded continuous,  $K(\t)$ is linear and bounded  for each $\t \in \mathbb{R}^m$ and absolutely integrable, then $F$ is a continuous function.  Let $({\bf b}_{k})$ be an arbitrary sequence in ${\mathrm R}_{\mathcal{G}}$, then there exist a subsequence $({\bf b}_{k_{l}})$ of $({\bf b}_{k})$  and a function
$f^{\ast} : {\mathbb R}^{m}\rightarrow \X$ such that the following limits hold
\begin{align}\label{love1234567801}
\lim_{l\rightarrow +\infty}\sup_{{\bf t}\in \mathcal{E}}||f\bigl({\bf t} +{\bf b}_{k_{l}})-f^{\ast}({\bf t}) ||_{\mathcal{X}}=0
\end{align}
and
\begin{align}\label{love12345678902}
\lim_{l\rightarrow +\infty}\sup_{{\bf t}\in \mathcal{E}}||f^{\ast}\bigl({\bf t}-{\bf b}_{k_{l}})-f({\bf t}) ||_{\mathcal{X}}=0\, ;
\end{align}
where $\mathcal{E}$ is a compact subset of $\mathbb{R}^m$. Let us define the function
$$
F^{\ast}({\bf t}):=\int_{\mathbb{R}^m}K({\bf t}-\eta)f^{\ast}(\eta)\, d\eta,\quad {\bf t}\in {\mathbb R}^{m}\, .
$$
Take $\mathcal{E}\subset \mathbb{R}^m$ a compact set; then, for any ${\bf t} \in \mathcal{E} $ we have
\begin{eqnarray*}
||F({\bf t}+{\bf b}_{k_{l}})-F^{\ast}({\bf t})||_{\mathcal{Y}} &=&|| \int_{\mathbb{R}^m}K(\eta)( f({\bf t}-\eta + {\bf b}_{k_{l}})-f^{\ast}({\bf t}-\eta))\, d\eta||_{\mathcal{Y}}\\
&\leq &   \int_{\mathbb{R}^m}||K(\eta)||_{\mathcal{L}(\mathcal{X},\mathcal{Y})}  \sup_{z\in \mathcal{E}_{\eta}}|| f(z+{\bf b}_{k_{l}})-f^{\ast}(z)||_{\mathcal{X}}d\eta\, ,
\end{eqnarray*}
where, $\mathcal{E}_{\eta}:=\mathcal{E} -{\eta}$ is a compact subset of $\mathbb{R}^m$. Using 
(\ref{love1234567801}) and the Lebesgue's Dominated Convergence Theorem, we conclude
$$\lim_{l \to \infty}\sup_{{\bf t}\in \mathcal{E}}||F({\bf t}+{\bf b}_{k_{l}})-F^{\ast}({\bf t})||_{\mathcal{Y}}=0\, .$$
Analogously, we have
$$\lim_{l \to \infty}\sup_{{\bf t}\in \mathcal{E}}||F^{\ast}({\bf t}-{\bf b}_{k_{l}})-F({\bf t})||_{\mathcal{Y}}=0\, .$$
\end{proof}

Based on the characterization theorems of compactly multi-almost automorphic functions presented in Section \ref{section char}, we derive the following corollary:
\begin{cor}\label{NewCor_002}
Let $\mathcal{G}$ be either $\mathbb{Z}^m$ or a dense subset of $\mathbb{R}^m$ and $f:\mathbb{R}^{m}\rightarrow \X$ be a bounded and ${\mathcal{G}}$-multi-almost automorphic function. Let 
$(K({\bf t}))_{{\bf t}\in \mathbb{R}^{n}}\subseteq \mathcal{L}(\X ; \mathcal{Y})$ be a strongly continuous operator family and 
$$\int_{\mathbb{R}^{m}}\| K({\bf t})\|_{\mathcal{L}(\X ; \mathcal{Y})}\, d {\bf t}\, < +\infty .$$
Define
\begin{equation}\label{rep021}
F({\bf t}):=\int_{\mathbb{R}^m}K({\bf t}-\eta)f(\eta)\, d\eta,\quad {\bf t}\in {\mathbb R}^{m}.
\end{equation}
If $F$ is uniformly continuous, then $F$ is compactly multi-almost automomorphic; that is $F \in \mathcal{K} AA(\mathbb{R}^m; \mathcal{Y})$.
\end{cor}


In what follows, we investigate the invariance of the bounded and compactly \(\mathrm{R}_{\mathcal{G}}\)-multi-almost automorphic function space in the setting where the kernel \( K \) in (\ref{rep01}) or (\ref{rep02}) depends on two variables, i.e., it takes the form \( K(\cdot, \cdot) \). To this end, we introduce the concept of a (compactly) \( (\mathrm{R}_{\mathcal{G}}, \mathcal{G}) \)-multi-Bi-almost automorphic function in the subsequent definition.

\begin{definition}\label{defBaa} A jointly continuous function $K:\mathbb{R}^m\times \mathbb{R}^m\times \X \to \mathcal{Y}$ is $(\mathrm{R}_{\mathcal{G}},\mathcal{B})$-multi-Bi-almost automorphic if for any $B \in \mathcal{B}$ and any sequence $({\bf b}_k)_{k \in \mathbb{N}} \in  \mathrm{R}_{\mathcal{G}}$, there exist a subsequence $({\bf b}_{k_l})\subset ({\bf b}_k)$ and a function $K^*:\mathbb{R}^m\times \mathbb{R}^m\times \X \to \mathcal{Y}$ such that
\begin{equation}\label{EqNew01}
\lim_{l \to +\infty} K({\bf t}+{\bf b}_{k_l}, {\bf s}+{\bf b}_{k_l} ; x)=K^*({\bf t}, {\bf s} ;x)\, ,
\end{equation}
and 
\begin{equation}\label{EqNew02}
\lim_{l \to +\infty}K^*({\bf t}-{\bf b}_{k_l} ,  {\bf s}-{\bf b}_{k_l} ;x)=K({\bf t}, {\bf s} ; x)\, ,
\end{equation}
hold pointwise for $({\bf t}, {\bf s}) \in \mathbb{R}^m \times  \mathbb{R}^m$ and uniformly for $x \in B$. 
If the limits in (\ref{EqNew01}) and (\ref{EqNew02}) are uniform on compact subsets of $\mathbb{R}^m \times \mathbb{R}^m$, then we say that $K$ is compactly $(\mathrm{R}_{\mathcal{G}},\mathcal{B})$-multi-Bi-almost automorphic. In case $\mathcal{G}=\mathbb{R}^m$ and ${\mathrm{R}}_{\mathbb{R}^m}={\mathrm{R}}_{\mathbb{R}^m}^a$, we say that $F$ is $\mathcal{B}$-multi-Bi-almost automorphic (resp. compactly $\mathcal{B}$-multi-Bi-almost automorphic).
\end{definition}

If, in Definition \ref{defBaa}, \( K: \mathbb{R}^m\times \mathbb{R}^m \to \mathcal{L}(\X ;  \mathcal{Y}) \) depends only on the variables \( (\t, \s) \in \mathbb{R}^m \times \mathbb{R}^m \), then the continuity of $K$ and the limits in (\ref{EqNew01}) and (\ref{EqNew02}), are understood in the next sense :
\begin{enumerate}
\item [({\bf SC})] The operator valued function \( K: \mathbb{R}^m \times \mathbb{R}^m \to \mathcal{L}(\X; \mathcal{Y}) \) is {\bf Strongly Continuous}, that is: if, for each \( x \in \X \), the function \( (\mathbf{t}, \mathbf{s}) \mapsto K(\mathbf{t}, \mathbf{s}) x \) is continuous.

In this case, the limits (\ref{EqNew01}) and (\ref{EqNew02}) are understood in the next sense
\begin{equation}\label{EqNew0101}
\lim_{l \to +\infty} K({\bf t}+{\bf b}_{k_l}, {\bf s}+{\bf b}_{k_l}) x=K^*({\bf t}, {\bf s}) x\, ,
\end{equation}
and 
\begin{equation}\label{EqNew0202}
\lim_{l \to +\infty}K^*({\bf t}-{\bf b}_{k_l} ,  {\bf s}-{\bf b}_{k_l}) x=K({\bf t}, {\bf s} ) x \, ,
\end{equation}
uniformly for $x \in B$. In this situation $K$ is called $S$-$(\mathrm{R}_{\mathcal{G}},\mathcal{B})$-multi-Bi-almost automorphic. If the limits in (\ref{EqNew0101}) and (\ref{EqNew0202}) are uniform on compact subsets of $\mathbb{R}^m \times \mathbb{R}^m$ then, we say that $K$ is compactly $S$-$(\mathrm{R}_{\mathcal{G}},\mathcal{B})$-multi-Bi-almost automorphic.

%
\end{enumerate} 

\begin{example}

Let $\mathcal{X}$ and $\mathcal{Y}$ be real Banach spaces, let $\mathcal{B}$ denote the family of all bounded subsets of $\mathcal{X}$, and let $A:\mathcal{X}\to\mathcal{Y}$ be a compact linear operator. Define the function
\[
F:\mathbb{R}^m\times\mathbb{R}^m\times\mathcal{X}\to\mathcal{Y}
\]
by
\[
F(\mathbf{t},\mathbf{s} ; x)=\frac{1}{e^{\lVert \mathbf{t}-\mathbf{s}\rVert}}\,A(x).
\]

It is straightforward to verify that $F$ is $\mathcal{B}$-multi-Bi-almost automorphic. However, $F$ is not $\mathcal{B}$-multi-almost automorphic. To see this, assume by contradiction that $F$ is $\mathcal{B}$-multi-almost automorphic and consider the sequence
\[
({\bf b}_k)_{k\in\mathbb{N}}
=\bigl((k,k,\ldots,k,0,0,\ldots,0)\bigr)_{k\in\mathbb{N}}
\subset\mathbb{R}^{m}\times \mathbb{R}^m\, ,
\]
where the entry $k$ appears $m$ times and the entry $0$ also appears $m$ times.

Then, for any subsequence $({\bf b}_{k_\ell})$ of $({\bf b}_k)$, we have
\[
\lim_{\ell\to\infty}
F\bigl((\mathbf{t},\mathbf{s})+{\bf b}_{k_\ell} ; x\bigr)
=0=:F^{*}(\mathbf{t},\mathbf{s} ; x),
\]
pointwise for $(\mathbf{t},\mathbf{s})\in\mathbb{R}^m\times\mathbb{R}^m$ and uniformly with respect to $x\in B$, for every $B\in\mathcal{B}$. Nevertheless, it is not possible to recover $F(\mathbf{t},\mathbf{s},x)$ by taking the limit as $\ell\to\infty$ of
\[
F^{*}\bigl((\mathbf{t},\mathbf{s})-{\bf b}_{k_\ell} ; x\bigr),
\]
which contradicts the definition of $\mathcal{B}$-multi-almost automorphy (see Definition \ref{eovako}).
\end{example}
%
%

The following definition generalizes to the multidimensional setting the notion introduced in the one-dimensional case in \cite{chavez2021almostaut}, where it was employed to study the existence and uniqueness of almost automorphic-type solutions to integral equations with advanced and delayed arguments.
\begin{definition}
We say that a jointly continuous function $K:\mathbb{R}^m\times\mathbb{R}^m\times \X \to \mathcal{Y}$ is $\lambda$-bounded if there exists a non negative function $\lambda:\mathbb{R}^m\times\mathbb{R}^m \to \mathbb{R}$ such that for every $\tau \in \mathbb{R}^m$ we have
\begin{equation}\label{Eeqqlambda}
||K(\t+\tau,\s+\tau ; x)||_{\mathcal{Y}} \leq \lambda({\bf t},{\bf s})\, ,
\end{equation}
where the inequality is for each $(\t,\s)\in \mathbb{R}^m\times\mathbb{R}^m$ and any $x\in \X$.\\
Also, if $K: \mathbb{R}^m \times \mathbb{R}^m \to \mathcal{L}(\mathcal{X}; \mathcal{Y})$ then, the inequality (\ref{Eeqqlambda}) is understood in the sense 
\begin{equation}\label{Eeqqlambda01}
||K(\t+\tau,\s+\tau) x||_{\mathcal{Y}} \leq \lambda({\bf t},{\bf s})||x||_{\mathcal{X}}\, ,
\end{equation}
for each $(\t,\s)\in \mathbb{R}^m\times\mathbb{R}^m$ and any $x\in \X$. 
\end{definition}

Note that, since for each $x\in \mathcal{X}$, there exists $B \in \mathcal{B}$ such that $x \in B$, then (\ref{Eeqqlambda}) is equivalent to
\begin{equation}\label{Eeqqlambda0011}
\sup_{x\in B}||K(\t+\tau,\s+\tau ; x)||_{\mathcal{Y}} \leq \lambda({\bf t},{\bf s})\, ,
\end{equation}
for each $(\t,\s)\in \mathbb{R}^m\times\mathbb{R}^m$. While, if $K: \mathbb{R}^m \times \mathbb{R}^m \to \mathcal{L}(\mathcal{X}; \mathcal{Y})$ and $\mathcal{B}$ is the collection of all bounded subsets of $\mathcal{X}$, then (\ref{Eeqqlambda01}) implies that
\begin{equation}\label{Eeqqlambda000111}
\sup_{x\in B}||K(\t+\tau,\s+\tau) x||_{\mathcal{Y}} \leq c(B)\lambda({\bf t},{\bf s})\, ,
\end{equation}
for some positive constant $c(B)$ and each $(\t,\s)\in \mathbb{R}^m\times\mathbb{R}^m$.


As it was proved for the one-dimensional case, see  \cite[Lemma 2.10]{chavez2024compact}, the inequality (\ref{Eeqqlambda}) also holds for the limit function $K^*$ of a $\lambda$-bounded and $(\mathrm{R}_{\mathcal{G}}, \mathcal{B})$-multi-Bi-almost auromorphic function $K$. That is, we have the following lemma

\begin{lem}\label{lemBou}
Let us suppose that $K:\mathbb{R}^m \times \mathbb{R}^m \times \X\to \mathcal{Y}$  is  $(\mathrm{R}_{\mathcal{G}}, \mathcal{B})$-multi-Bi-almost automorphic and $\lambda$-bounded. Then, its limit function $K^*:\mathbb{R}^m\times \mathbb{R}^m\times \X\to \mathcal{Y}$ (see definition \ref{defBaa}) is also $\lambda$-bounded; that is, for each $(\t,\s)\in \mathbb{R}^m\times\mathbb{R}^m$ and any $x\in \X$ we have
$$||K^*(\t+\tau,\s+\tau ; x)||_{\mathcal{Y}} \leq \lambda({\bf t},{\bf s}),\, \,  \forall \tau \in \mathbb{R}^m\, .$$
\end{lem}

Analogously, we have
\begin{lem}\label{lemBou02}
Let us suppose that $K:\mathbb{R}^m \times \mathbb{R}^m \to \mathcal{L}(\mathcal{X}; \mathcal{Y})$  is  $S$-$(\mathrm{R}_{\mathcal{G}}, \mathcal{B})$-multi-Bi-almost automorphic and $\lambda$-bounded. Then, its limit function $K^*:\mathbb{R}^m\times \mathbb{R}^m \to \mathcal{L}(\mathcal{X}; \mathcal{Y})$ (see definition \ref{defBaa}) is also $\lambda$-bounded; that is, for each $(\t,\s)\in \mathbb{R}^m\times\mathbb{R}^m$ and any $x\in \X$ we have
$$||K^*(\t+\tau,\s+\tau) x||_{\mathcal{Y}} \leq \lambda({\bf t},{\bf s})||x||_{\mathcal{X}},\, \,  \forall \tau \in \mathbb{R}^m\, .$$
\end{lem}



\noindent Let us introduce the next set
$$\mathrm{R}^{\mathcal{G}}_{\Delta_{{\mathbb{R}^m \times \mathbb{R}^m}}} :=\{ (a_{n}=(b_n,b_n))_{n\in \mathbb{N}}\, :\, (b_n)_{n\in\mathbb{N}} \in  \mathrm{R}_{\mathcal{G}}  \} \subset \Delta_{\mathbb{R}^m \times \mathbb{R}^m} \, ,$$
where $$\Delta_{\mathbb{R}^m \times \mathbb{R}^m} : =\{ (\t, \t),\ : \, \t \in \mathbb{R}^m\}\, , $$
is the diagonal of $\mathbb{R}^m \times \mathbb{R}^m $.

Note that, if  $\mathcal{G}=\mathbb{R}^m$ and $\mathrm{R}_{\mathbb{R}^m}=\mathrm{R}_{\mathbb{R}^m}^a$ (is the set of all the sequences in $\mathbb{R}^m$) then, $\mathrm{R}^{\mathcal{G}}_{\Delta_{{\mathbb{R}^m \times \mathbb{R}^m}}}=\mathrm{R}^{a}_{\Delta_{{\mathbb{R}^m \times \mathbb{R}^m}}}$ .
\begin{tw}\label{NewTTHHMM01}
Let $K:\mathbb{R}^m\times \mathbb{R}^m\times \X \to \mathcal{Y}$ be continuous. If $K$ is compactly $(\mathrm{R}_{\mathcal{G}},\mathcal{B})$-multi-Bi-almost automorphic, then
\begin{enumerate}
\item [a).] $K$ is compactly $(\mathrm{R}^{\mathcal{G}}_{\Delta_{{\mathbb{R}^m \times \mathbb{R}^m}}}, \mathcal{B} )$-multi-almost automorphic.

\item [b).] If $\mathrm{R}_{\mathcal{G}}$ is the set of sequences such that all of its subsequences are also in $\mathrm{R}_{\mathcal{G}}$, $K$ is bounded and continuous in $\mathbb{R}^m\times \mathbb{R}^m$, uniformly for any $B\in {\mathcal B}$, then $K$ is $(\mathrm{R}^{\mathcal{G}}_{\Delta_{{\mathbb{R}^m \times \mathbb{R}^m}}}, \mathcal{B} )$-uniformly continuous.
\end{enumerate}
\end{tw}

\begin{proof} Let us prove item $a)$. 
Let the sequence $(a_n) \in \mathrm{R}^{\mathcal{G}}_{\Delta_{{\mathbb{R}^m \times \mathbb{R}^m}}}$, then
\[
(a_n=(b_n,b_n))_{n\in\mathbb{N}},
\qquad \text{where } (b_n)\in\mathcal{R}_\mathcal{G}.
\]

Since $K$ is compactly $(\mathcal{R}_\mathcal{G},\mathcal{B})$-multi-Bi-almost automorphic,
there exist a subsequence $(c_n)\subset (b_n)$ and a function
\[
K^*:\mathbb{R}^m\times\mathbb{R}^m\times \X\longrightarrow \mathcal{Y}
\]
such that, for  any $B\in \mathcal{B}$, we have
\begin{equation}\label{EEEQQQQNNN01}
\lim_{n\to\infty}
\sup_{\substack{(\t,\s)\in \mathcal{E}}}\sup_{\substack{x \in B}}
\left\|
K(\t+c_n,\s+c_n;x)-K^*(\t,\s;x)
\right\|_{\mathcal{Y}}=0
\end{equation}
and
\begin{equation}\label{EEEQQQQNNN02}
\lim_{n\to\infty}
\sup_{\substack{(\t,\s)\in \mathcal{E}}}\sup_{\substack{x \in B}}
\left\|
K^*(\t-c_n,\s-c_n;x)-K(\t,\s;x)
\right\|_{\mathcal{Y}}=0,
\end{equation}
where $\mathcal{E}\subset\mathbb{R}^m\times\mathbb{R}^m$ is any compact set.

Now, let the sequence $(d_n=(c_n,c_n))_{n\in \mathbb{N}}$, then $(d_n)$ is a subsequence of $(a_n)$ and, defining ${\bf z}=(\t,\s)$ we have, from (\ref{EEEQQQQNNN01}) and (\ref{EEEQQQQNNN02}), that

\[
\lim_{n\to\infty}
\sup_{\substack{{\bf z}\in \mathcal{E}}}\sup_{\substack{x \in B}}
\left\|
K({\bf z}+d_n;x)-K^*({\bf z};x)
\right\|_{\mathcal{Y}}=0
\]
and
\[
\lim_{n\to\infty}
\sup_{\substack{{\bf z}\in \mathcal{E}}}\sup_{\substack{x \in B}}
\left\|
K^*({\bf z}-d_n;x)-K(\t,\s;x)
\right\|_{\mathcal{Y}}=0\, .
\]
Thus, $K$ is compactly $(\mathrm{R}^{\mathcal{G}}_{\Delta_{{\mathbb{R}^m \times \mathbb{R}^m}}}, \mathcal{B} )$-multi-almost automorphic.

The proof of item $b)$ 
follows from Theorem \ref{ThRKAA}.
\end{proof}

Analogously, we have the next result for the case of pointwise almost automorphy
\begin{tw}\label{NewTTHHMM02002}
Let $K:\mathbb{R}^m\times \mathbb{R}^m\times \X \to \mathcal{Y}$ be bounded and continuous. 
If $K$ is  $(\mathrm{R}_{\mathcal{G}},\mathcal{B})$-multi-Bi-almost automorphic then, 
$K$ is  $(\mathrm{R}^{\mathcal{G}}_{\Delta_{{\mathbb{R}^m \times \mathbb{R}^m}}}, \mathcal{B} )$-multi-almost automorphic.
\end{tw}

For the following theorem, the proof of part $a)$ is analogous to that of Theorem \ref{NewTTHHMM01}, while the proof of part $b)$ follows similarly to that of Theorem \ref{ThRKAA}. Therefore, we omit its proof.
\begin{tw}\label{THMDD}
Let $K: \mathbb{R}^m \times \mathbb{R}^m \to \mathcal{L}(\mathcal{X}; \mathcal{Y})$ be strongly continuous and compactly $S$-$(\mathrm{R}_{\mathcal{G}}, \mathcal{B})$-multi-Bi-almost automorphic. 
Then
\begin{enumerate}
\item [a)] $K$ is compactly $S$-$(\mathrm{R}^{\mathcal{G}}_{\Delta_{{\mathbb{R}^m \times \mathbb{R}^m}}},\mathcal{B})$-multi almost automorphic.
\item [b)] If  $\mathrm{R}_{\mathcal{G}}$ is such that all of its subsequences are also in $\mathrm{R}_{\mathcal{G}}$ and $K$ is $\lambda$-bounded 
then, $K$ is $(\mathrm{R}^{\mathcal{G}}_{\Delta_{{\mathbb{R}^m \times \mathbb{R}^m}}},\mathcal{B})$-uniformly continuous.
\end{enumerate}

\end{tw}



%

%


\begin{tw}\label{TeoComb}
Let $\mathrm{R}_{\mathcal{G}}$ be the set of sequences such that all of its subsequences are also in $\mathrm{R}_{\mathcal{G}}$ and $\mathcal{B}$ be the collection of all bounded sets in $\mathcal{X}$. Let the operator valued function $K:\mathbb{R}^m \times \mathbb{R}^m\to \mathcal{L}(\X; \mathcal{Y})$ be strongly continuous, compactly $S$-$(\mathrm{R}_{\mathcal{G}},\mathcal{B})$-multi-Bi-almost automorphic and $\lambda$-bounded with $\lambda({\bf t},{\bf s})=\phi({\bf t}-{\bf s})$  and $\phi \in L^1(\mathbb{R}^m)$. 
Then, the operator $\Gamma$, defined by
$$\Gamma u({\bf t}):= \int_{\mathbb{R}^m} K({\bf t},\eta)u(\eta)d\eta\, ,$$
maps the Banach space $\mathcal{K} \mathrm{R}_{\mathcal{G}} AA(\mathbb{R}^m ; \mathcal{X})\cap BC(\mathbb{R}^m ; \mathcal{X})$ to the Banach space $\mathcal{K} \mathrm{R}_{\mathcal{G}} AA(\mathbb{R}^m ;  \mathcal{Y})\cap BC(\mathbb{R}^m ; \mathcal{Y})$. 
\end{tw}

\begin{proof} Let $u \in \mathcal{K} \mathrm{R}_{\mathcal{G}} AA(\mathbb{R}^m; \X)\cap BC(\mathbb{R}^m ; \mathcal{X})$ and define   $v:=\Gamma u$.  First of all, we note that $v$ is a well defined, continuous and bounded  function. On the other hand, 
since $K(\cdot,\cdot)$ is compactly $S$-$(\mathrm{R}_{\mathcal{G}},\mathcal{B})$-multi-Bi-almost automorphic and $u$ is compactly $\mathrm{R}_{\mathcal{G}}$-multi-almost automorphic  
we can ensure that given any $B\in \mathcal{B}$ and any sequence $({\bf b}_k)\subset \mathrm{R}_{\mathcal{G}}$ there exist a subsequence $({\bf b}_{k_l})\subseteq ({\bf b}_k)$ and functions $K^*(\cdot,\cdot)$ and $u^*(\cdot)$ such that the following limits, uniformly on compact subsets of $\mathbb{R}^m\times \mathbb{R}^m$ and uniformly for $x\in B$, hold
$$\lim\limits_{l\to +\infty} K(\t+{\bf b}_{k_l},\s+{\bf b}_{k_l})x=K^*(\t,\s)x,\quad \lim\limits_{l\to +\infty} K^*(\t-{\bf b}_{k_l},\s-{\bf b}_{k_l})x=K(\t,\s)x;$$
and also the following limits, uniformly on compact subsets of $\mathbb{R}^m$ hold
$$\lim\limits_{l\to +\infty} u(\t+{\bf b}_{k_l})=u^*(\t),\quad \lim\limits_{l\to +\infty} u^*(\t-{\bf b}_{k_l})=u(\t).$$
Now, let $v^*:\mathbb{R}^m \to \mathcal{Y}$ be the function defined as follows: for each $\t \in \mathbb{R}^m$,
$$v^*(\t):=\int_{\mathbb{R}^m} K^*(\t,\eta)u^*(\eta)d\eta\, .$$
It follows from the hypothesis and Lemma \ref{lemBou} that $v^*$ is well defined. 

In what follows, we prove that the next limits
\begin{equation}\label{Eq001}
\lim\limits_{l\to +\infty} v(\t+b_{k_l})=v^*(\t),\, \, \lim_{l\to \infty}v^*(\t-b_{k_l})=v(\t)\, ,
\end{equation}
hold uniformly for ${\bf t}$ in compact subsets of $\mathbb{R}^m$. In fact, let $E$ be a compact subset of $\mathbb{R}^m$ then, for each $\eta \in \mathbb{R}^m$ the set $E_{\eta}:=E+ \eta$ is a compact subset of $\mathbb{R}^m$; also, since $u^*$ is bounded, then $E^*_{\eta}:=u^*(E+\eta)$ is a bounded set in $\X$ (contained in the ball of center zero and radius $||u||_{\infty}$), i.e. $E^*_{\eta} \in \mathcal{B}$. Then, for each $\t \in E$ and each $\eta \in \mathbb{R}^m$, we have
\begin{eqnarray}\label{INEQUALITY01}
& || \left(  K(\t+{\bf b}_{k_l}, \t+\eta+{\bf b}_{k_l})-K^*(\t,\t+\eta) \right) u^*(\t+\eta) ||_{\mathcal{Y}}& \leq \nonumber  \\ 
&\displaystyle\sup_{x\in E^*_{\eta}}|| \left(  K(\t+{\bf b}_{k_l}, \t+\eta+{\bf b}_{k_l})-K^*(\t,\t+\eta) \right) x ||_{\mathcal{Y}}& \leq \nonumber \\
&
\displaystyle\sup_{(\t,\omega) \in E\times E_{ \eta }}\sup_{x\in E^*_{\eta}}|| \left(  K(\t+{\bf b}_{k_l}, \omega+{\bf b}_{k_l})-K^*(\t,\omega) \right) x||_{\mathcal{Y}}& \nonumber \, ,
\end{eqnarray}
which implies, by $\lambda$-boundedness of $K$, that
$$|| \left(  K(\t+{\bf b}_{k_l}, \t+\eta+{\bf b}_{k_l})-K^*(\t,\t+\eta) \right) u^*(\t+\eta) ||_{\mathcal{Y}}\leq 2\phi(-\eta)  ||u||_{\infty}\, $$
and also that
$$\lim_{l\to \infty}\sup_{\t \in E}|| \left(  K(\t+{\bf b}_{k_l}, \t+\eta+{\bf b}_{k_l})-K^*(\t,\t+\eta) \right) u^*(\t+\eta) ||_{\mathcal{Y}}=0\, .$$

%
Therefore, by using the Lebesgue's Dominated Convergence Theorem, we obtain
%
\begin{eqnarray*}
\sup_{\t \in E}||v(\t+{\bf b}_{k_l})-v^*(\t)||_{\mathcal{Y}}& \leq &\sup_{\t \in E}|| \int_{\mathbb{R}^m} \left(  K(\t+{\bf b}_{k_l}, \eta+{\bf b}_{k_l})-K^*(\t,\eta) \right) u^*(\eta)d \eta ||_{\mathcal{Y}}\\
&+& 
\sup_{\t \in E}|| \int_{\mathbb{R}^m} K(\t+{\bf b}_{k_l}, \eta+{\bf b}_{k_l}) \left( u(\eta+{\bf b}_{k_l})-u^*(\eta) \right)d \eta||_{\mathcal{Y}} \\
&\leq & \int_{\mathbb{R}^m} \sup_{\t \in E}|| \left(  K(\t+{\bf b}_{k_l}, \t+\eta+{\bf b}_{k_l})-K^*(\t,\t+\eta) \right) u^*(\t+\eta) ||_{\mathcal{Y}}d \eta  \\ 
&+& 
\int_{\mathbb{R}^m} \phi(\eta) \, 
\sup_{z\in E-\eta }|| u(z+{\bf b}_{k_{l}})-u^{\ast}(z)||_{\mathcal{X}}d\eta\, \\
&\to  & 0,\, \, \,  {\rm as }\, \, \, \,  l\to \infty\, .
\end{eqnarray*}
That is
$$\lim_{l\to \infty}\sup_{{\bf t}\in E}||v(\t+{\bf b}_{k_l})-v^*(\t)||_{\mathcal{Y}}=0\, .$$
By an entirely analogous argument, we have
$$\lim_{l\to \infty}\sup_{{\bf t}\in E}||v^*(\t-{\bf b}_{k_l})-v(\t)||_{\mathcal{Y}} =0\, . $$
\end{proof}

The  next result is analogous to the previous one, but for pointwise multi-almost automorphy. We reproduce the proof for the sake of completeness.


\begin{tw}\label{TeoAAInvar}
Let $\mathrm R_{\mathcal{G}}$ be the set of sequences such that all of its subsequences are also in $\mathrm R_{\mathcal G}$ and $\mathcal B$ be the collection of all bounded sets in $\mathcal X$. Let the operator valued function $K:\mathbb{R}^m\times\mathbb R^m\to\mathcal L(\mathcal X;\mathcal
Y)$ be strongly continuous, $S$-$(\mathrm{R}_{\mathcal G},\mathcal
B)$-multi-Bi-almost automorphic and $\lambda$-bounded with
$\lambda(\mathbf t,\mathbf s)=\phi(\mathbf t-\mathbf s)$, $\phi\in
L^1(\mathbb R^m)$. Then the operator $\Gamma$, defined by
$$\Gamma u(\mathbf t):=\int_{\mathbb R^m}K(\mathbf t,\eta)u(\eta)\,d\eta,$$
maps the Banach space $\mathrm{R}_{\mathcal{G}}AA(\mathbb R^m;\mathcal
X)\cap BC(\mathbb R^m;\mathcal X)$ into the Banach space $\mathrm
R_{\mathcal G}AA(\mathbb{R}^m;\mathcal{Y})\cap BC(\mathbb{R}^m;\mathcal{Y})$.
\end{tw}
 
\begin{proof}
Let $u\in\mathrm{R}_{\mathcal{G}}AA(\mathbb R^m;\mathcal X)\cap BC(\mathbb
R^m;\mathcal X)$ and set $v:=\Gamma u$. Then, as in the previous Theorem, we have that $v$ is a well defined, bounded and continuous function.
%
%
Now, given $(\mathbf b_k)\subset\mathrm R_{\mathcal
G}$, by hypothesis there exist a subsequence $(\mathbf
b_{k_l})\subseteq(\mathbf b_k)$ and functions $K^*(\cdot,\cdot)$ and  $u^*(\cdot)$ such
that, for each $(\mathbf t,\mathbf s)\in\mathbb R^m\times\mathbb
R^m$ and $\mathbf t\in\mathbb R^m$ respectively, and uniformly for $x\in B$, for every $B\in\mathcal B$, we have
\[\lim_{l\to+\infty}K(\mathbf t+\mathbf b_{k_l},\mathbf s+\mathbf b_{k_l})x=
K^*(\mathbf t,\mathbf s)x,\qquad
\lim_{l\to+\infty}K^*(\mathbf t-\mathbf b_{k_l},\mathbf s-\mathbf b_{k_l})x=
K(\mathbf t,\mathbf s)x, \]
\[\lim_{l\to+\infty}u(\mathbf t+\mathbf b_{k_l})=u^*(\mathbf t),\qquad
\lim_{l\to+\infty}u^*(\mathbf t-\mathbf b_{k_l})=u(\mathbf t). \]
Let the well defined and bounded function $v^*$ defined by
$$v^*(\mathbf t):=\int_{\mathbb R^m}K^*(\mathbf t,\eta)u^*(\eta)\,d\eta,
\qquad\mathbf t\in\mathbb R^m\, .$$
 
\noindent Let us fix $\mathbf t\in\mathbb R^m$, then we have
\begin{align*}
v(\mathbf t+\mathbf b_{k_l})-v^*(\mathbf t)
&=\int_{\mathbb R^m}\!\big(K(\mathbf t+\mathbf b_{k_l},\eta+\mathbf
b_{k_l})-K^*(\mathbf t,\eta)\big)u^*(\eta)\,d\eta\\
&\quad+\int_{\mathbb R^m}\!K(\mathbf t+\mathbf b_{k_l},\eta+\mathbf b_{k_l})
\big(u(\eta+\mathbf b_{k_l})-u^*(\eta)\big)\,d\eta\\
&=:A_l(\mathbf t)+B_l(\mathbf t).
\end{align*}
where
$$A_l(\t)=\int_{\mathbb R^m}\!\big(K(\mathbf t+\mathbf b_{k_l},\eta+\mathbf
b_{k_l})-K^*(\mathbf t,\eta)\big)u^*(\eta)\,d\eta\, ,$$
and 
$$B_l(\t)=\int_{\mathbb R^m}\!K(\mathbf t+\mathbf b_{k_l},\eta+\mathbf b_{k_l})
\big(u(\eta+\mathbf b_{k_l})-u^*(\eta)\big)\,d\eta \, .$$
It is not difficult to see that $A_l(\t) \to 0$ as $l\to \infty$ and also that $B_l(\t) \to 0$ as $l\to \infty$.\\

\noindent Therefore, for each $\t \in \mathbb{R}^m$ we have 
$$\lim_{l\to \infty}v(t+\mathbf{b}_{k_l})=v^*(\t)\, .$$
The reverse pointwise limit 
$$\lim_{l\to \infty}v^*(t-\mathbf{b}_{k_l})=v(\t)\, ,\, \, \t \in \mathbb{R}^m$$
follows analogously.
\end{proof}

\begin{tw}\label{NewThmappl}
Let $\mathrm{R}_{\mathcal{G}}$ be the set of sequences such that all of its subsequences are also in $\mathrm{R}_{\mathcal{G}}$ and $\mathcal{B}$ be the collection of all the bounded sets in $\mathcal{X}$. Let the operator valued function $K:\mathbb{R}^m \times \mathbb{R}^m\to \mathcal{L}(\X; \mathcal{Y})$ be strongly continuous, 
compactly $S$-$\mathcal{B}$-multi-Bi-almost automorphic and is $\lambda$-bounded with $\lambda({\bf t},{\bf s})=\phi({\bf t}-{\bf s})$  and $\phi \in L^1(\mathbb{R}^m)$. 
Then, the operator $\Gamma$, defined by
$$\Gamma u({\bf t}):= \int_{\mathbb{R}^m} K({\bf t},\eta)u(\eta)d\eta\, ,$$
maps the space $\mathrm{R}_{\mathcal{G}} AA(\mathbb{R}^m ; \mathcal{X})\cap B UC(\mathbb{R}^m ; \mathcal{X})$ to the space $\mathrm{R}_{\mathcal{G}} AA(\mathbb{R}^m ; \mathcal{Y})\cap B UC(\mathbb{R}^m ; \mathcal{Y})$.
\end{tw}
\begin{proof}
Using the hypothesis and arguing as in the proof of Theorem \ref{TeoComb}, we see that if $ u \in \mathrm{R}_{\mathcal{G}} AA(\mathbb{R}^m, \mathcal{X})\cap B UC(\mathbb{R}^m ; \mathcal{X})$, then $v:= \Gamma u \in \mathrm{R}_{\mathcal{G}} AA(\mathbb{R}^m, \mathcal{Y})\cap BC(\mathbb{R}^m ; \mathcal{Y})$. It only rest to prove that $\Gamma u$ is uniformly continuous. 
From Theorem \ref{THMDD} we have that $K$ is $(\Delta_{\mathbb{R}^m \times \mathbb{R}^m},\mathcal{B})$-uniformly continuous. Also note that, since $u$ is a bounded function, then the image of $u$ is a bounded set, i.e. the set $B^{u}:=\{ u(z)\, : \, z\in \mathbb{R}^m \} \in \mathcal{B}$. \\
Now, let $(\t_k),(\s_k)$ be two sequences in $\mathbb{R}^m$ such that $||\t_k-\s_k|| \to 0$ when $k\to +\infty$, then
\begin{eqnarray*}
||v(\t_k)-v(\s_k)||_{\mathcal{Y}}&\leq & \int_{\mathbb{R}^m}||\left(  K(\t_k, \eta+\t_k)-K(\s_k,\eta+\s_k) \right) u(\eta+\t_k)||_{\mathcal{Y}}d\eta \\
&+& \int_{\mathbb{R}^m}||K(\s_k,\eta+\s_k)\left( u(\eta+\t_k)-u(\eta+\s_k)\right) ||_{\mathcal{Y}}d\eta\\
&=:&I_1(k)+I_2(k)\, ,
\end{eqnarray*}
where
$$I_1(k):= \int_{\mathbb{R}^m}||\left(  K(\t_k, \eta+\t_k)-K(\s_k,\eta+\s_k) \right) u(\eta+\t_k)||_{\mathcal{Y}}d\eta \, ,$$
$$I_2(k):= \int_{\mathbb{R}^m}||K(\s_k,\eta+\s_k)\left( u(\eta+\t_k)-u(\eta+\s_k)\right) ||_{\mathcal{Y}}d\eta\, .$$
Let us show that $I_1(k) \to 0$, when $k\to \infty$. First observe that since $||\t_k-\s_k||\to 0$, then the sequences $((\t_k,\t_k)),\, ((s_k,s_k))\in \Delta_{\mathbb{R}^m\times \mathbb{R}^m}$ satisfies $||(\t_k,\t_k)-(\s_k,\s_k)||\to 0$ as $k\to \infty$, and since $K$ is $(\Delta_{\mathbb{R}^m \times \mathbb{R}^m},\mathcal{B})$-uniformly continuous then, for any $B\in \mathcal{B}$ and each $(\xi,\eta)\in \mathbb{R}^m\times \mathbb{R}^m$, we have
$$\lim_{k\to \infty}\sup_{x\in B}||(K(\xi+\t_k, \eta+\t_k)-K(\xi+\s_k,\eta+\s_k))x||_{Y}=0\, .$$
Then, using the Lebesgue's Dominated Convergence Theorem, we have
\begin{eqnarray*}
I_1(k)&=&\int_{\mathbb{R}^m}||\left(  K(\t_k, \eta+\t_k)-K(\s_k,\eta+\s_k) \right) u(\eta+\t_k)||_{\mathcal{Y}}d\eta\\
&\leq & \int_{\mathbb{R}^m}\sup_{x\in B^u} ||\left(  K(\t_k, \eta+\t_k)-K(\s_k,\eta+\s_k) \right) x||_{\mathcal{Y}}d\eta\\
&\to & 0,\, k \to \infty \, .
\end{eqnarray*}
For the second integral, since $u$ is uniformly continuous, we have
\begin{eqnarray*}
I_2(k)&\leq & \int_{\mathbb{R}^m}\phi(-\eta)||u(\eta+\t_k)-u(\eta+\s_k)||_{\mathcal{X}}d\eta\, \\
&\to & 0,\, k \to \infty \, .
\end{eqnarray*}
Therefore, 
$$\lim_{k\to \infty} ||v(\t_k)-v(\s_k)||_{\mathcal{Y}}=0\, .$$
\end{proof}

\begin{cor}
Let $\mathcal{B}$ be the collection of all bounded sets in $\mathcal{X}$ and let the operator valued function $K:\mathbb{R}^m \times \mathbb{R}^m\to \mathcal{L}(\X; \mathcal{Y})$ be strongly continuous, compactly $S$-$\mathcal{B}$-multi-Bi-almost automorphic and $\lambda$-bounded, where $\lambda({\bf t},{\bf s})=\phi({\bf t}-{\bf s})$  and $\phi \in L^1(\mathbb{R}^m)$. 
Then, the operator $\Gamma$, defined by
$$\Gamma u({\bf t}):= \int_{\mathbb{R}^m} K({\bf t},\eta)u(\eta)d\eta\, ,$$
maps the Banach space $\mathcal{K} AA(\mathbb{R}^m ; \mathcal{X})$ to the Banach space $\mathcal{K} AA(\mathbb{R}^m ; \mathcal{Y})$. 

\end{cor}

Using the previous results and the characterization theorems of the space of compactly multi-almost automorphic functions presented in section \ref{section char}, it is not difficult to prove the next corollary
\begin{cor}\label{NewCor_001}
Let $\mathcal{G}$ be either $\mathbb{Z}^m$ or a dense subset of $\mathbb{R}^m$ and 
$\mathcal{B}$ the collection of all bounded sets in $\mathcal{X}$. Let the operator valued function $K:\mathbb{R}^m \times \mathbb{R}^m\to \mathcal{L}(\X; \mathcal{Y})$ be strongly continuous,  $S$-$(\mathcal{G}, \mathcal{B})$-multi-Bi-almost automorphic and $\lambda$-bounded, where $\lambda({\bf t},{\bf s})=\phi({\bf t}-{\bf s})$  and $\phi \in L^1(\mathbb{R}^m)$.
Then, if $u$ is bounded and belongs to $\mathcal{G} AA(\mathbb{R}^m ; \mathcal{X})$, and  $\Gamma u$ is uniformly continuous, where
$$\Gamma u({\bf t}):= \int_{\mathbb{R}^m} K({\bf t},\eta)u(\eta)d\eta $$
then, $\Gamma u \in \mathcal{K} AA(\mathbb{R}^m ; \mathcal{Y})$.

\end{cor}
The proof of the following result
is analogous to that of Theorem \ref{NewThmappl}, so we omit the details.


\begin{tw}\label{InvConv1}
Let $\mathcal{B}$ be the collection of all bounded sets in $\mathcal{X}$ and let the operator valued function $K:\mathbb{R}^m \times \mathbb{R}^m\to \mathcal{L}(\X; \mathcal{Y})$ be strongly continuous, compactly $S$-$\mathcal{B}$-multi-Bi-almost automorphic and $\lambda$-bounded, where $\lambda({\bf t},{\bf s})=\phi({\bf t}-{\bf s})$  and $\phi \in L^1([0,\infty)^m)$. Then, the operator $\Pi$ defined by 
\begin{equation}\label{Operator22}
\Pi u ({\bf t}):=\int_{-\infty}^{{\bf t}} K({\bf t},\eta)u(\eta)d\eta\, ,
\end{equation}
maps the Banach space $\mathcal{K}AA(\mathbb{R}^m; \mathcal{X})$ to the Banach space $\mathcal{K}AA(\mathbb{R}^m; \mathcal{Y})$.
\end{tw}
Finally, we point out that applications of Theorems \ref{TeoComb} and \ref{InvConv1} to ordinary differential equations with exponential dichotomy can be found, for instance, in \cite{chavez2024compact}.



\section{Applications to Poisson's and heat equations}\label{section appl}
In this section, we study the almost automorphic regularity of Poisson's and heat equations.

\subsection{Compact almost automorphy of the solution to Poisson's equation}
\begin{definition}
Let $p\geq 1$ and $ f\in L^{p}_{loc}(\mathbb{R}^{m},\mathbb{R})$. A function $u\in L^{p}_{loc}(\mathbb{R}^{m},\mathbb{R})$ is solution in the sense of distributions of the Poisson's equation 
$$\Delta u=f,$$
if
\begin{equation}
\notag\int_{\mathbb{R}^{m}}u(x)\Delta\phi(x)dx=\int_{\mathbb{R}^{m}} f(x)\phi(x)dx,\hspace{0.5cm}\forall\phi\in C^{\infty}_{0}(\mathbb{R}^{m}).
\end{equation}
\end{definition}
The following is a well known result within the theory of distributional solutions for the Laplace equation
\begin{tw} \label{Thmf01} \cite{Jost02}
\begin{enumerate}
\item (Weyl) If $u\in L^{1}_{loc}(\mathbb{R}^{m},\mathbb{R})$ and $u$ is a solution in the sense of distributions of $\Delta u=0$, then $u$ is harmonic, i.e., $u\in C^{\infty}(\mathbb{R}^{m})$ and $\Delta u=0$ in $\mathbb{R}^{m}$.
\item (Lioville) If $u:\mathbb{R}^{m}\to \mathbb{R}$ is harmonic and bounded, then $u$ is constant.
\end{enumerate}
\end{tw}


In the next result, it is established the compact $\mathrm{R}_{\mathcal{G}}$-multi-almost automorphic regularity of Poisson’s equation. 
%
The proof is analogous to that presented in \cite{CHAKPPINTO2023}, but we give here the full proof for the sake of completeness.
%

\begin{tw}\label{main theorem -1} 
 Let $f: \mathbb{R}^m\to \mathbb{R}$ be bounded, continuous and ${\mathrm{R}}_{{\mathcal{G}}}$-multi-almost automorphic. If $u$ is bounded, continuous and solution in the sense of distributions of Poisson's equation
\begin{equation}
\notag\Delta u=f\, ,
\end{equation}
then $u\in \mathcal{K} {\mathrm{R}}_{{\mathcal{G}}}AA(\mathbb{R}^{m} ; \mathbb{R})$.
\end{tw}

\begin{proof}
Since \( u \) and  \( f \) are bounded continuous functions, it follows from the Lipschitz regularity of distributional solutions to Poisson’s equation, see for instance \cite{Courant62}, that \( u \) is globally Lipschitz on \( \mathbb{R}^m \), and hence uniformly continuous on \( \mathbb{R}^m \).

Let $( {\bf s}^{\prime}_{n}) \in {\mathrm{R}}_{{\mathcal{G}}}$, since $f\in {\mathrm{R}}_{{\mathcal{G}}}AA(\mathbb{R}^{m},\mathbb{R})$, there exist a subsequence $( {\bf s}_{n}) \subset ( {\bf s}^{\prime}_{n})$ and a function $g:\mathbb{R}^{m}\to \mathbb{R}$ such that the following pointwise limits holds 
\begin{equation}
\notag\lim_{n\to \infty}f(\t+{\bf s}_{n})=g(\t)\hspace{1cm}{\rm and } \hspace{1cm}\lim_{n\to \infty}g(\t-{\bf s}_{n})=f(\t)\, .
\end{equation}

Let $K\subset\mathbb{R}^{m}$ be a compact set. Now, reasoning as in the proof of Theorem \ref{ThRKAA Partial converse} (where an application of the Arzelà-Ascoli Theorem was involved), 
%
%
%
there exist a 
subsequence $({\bf s}_{l_n})$ of $({\bf s}^{\prime}_n)$ and a function $h: \mathbb{R}^m \to \mathbb{R}$ such that the following limit hold
\begin{equation}
\notag\lim_{n\to \infty}\sup_{\t \in K} |u(\t+{\bf s}_{l_{n}})-h(\t)|=0\, .
\end{equation}
Note that $h \in BUC(\mathbb{R}^m ; \mathbb{R})$. Also, following the same argument, there exist a subsequence $({\bf p}_n)$ of $({\bf s}_{l_n})$ and a bounded and uniformly continuous function $k:\mathbb{R}^{m}\to \mathbb{R}$ such that on any compact set $K\subset
\mathbb{R}^m$ the following limit holds
\begin{equation}
\notag \lim_{n\to \infty} \sup_{\t \in K} | h(\t-{\bf p}_n)-k(\t)|=0\, .\hspace{1cm}
\end{equation}
It only rest to prove that $k=u$.

For each \(n \in \mathbb{N}\), define the sequences \(f_n, u_n : \mathbb{R}^m \to \mathbb{R}\) by \(f_n(\t) = f(\t + {\bf p}_n)\) and \(u_n(\t) = u(\t + {\bf p}_n)\). Then, 
\begin{equation}\label{Eeqnn0001}
\Delta u_n = f_n
\end{equation}
holds in the sense of distributions. Indeed, let \(\phi \in C_c^\infty(\mathbb{R}^m)\), then we have
\[
\int_{\mathbb{R}^m} u_n(\t) \Delta \phi(\t) \, d \t = \int_{\mathbb{R}^m} u(\t + {\bf p}_n) \Delta \phi(\t) \, d \t = \int_{\mathbb{R}^m} u(\t) \Delta \phi(\t - {\bf p}_n) \, d \t.
\]
Setting \(\phi_n(\t) := \phi(\t - {\bf p}_n)\), then for each $n$ we have \(\phi_n \in C_c^\infty(\mathbb{R}^m)\). Since \(\Delta u = f\) in the sense of distributions, then
\[
\int_{\mathbb{R}^m} u(\t) \Delta \phi_n(\t) \, d \t = \int_{\mathbb{R}^m} f(\t) \phi_n(\t) \, d \t.
\]
It follows that
\begin{equation}\label{EeQqNn0011}
\int_{\mathbb{R}^m} u_n(\t) \Delta \phi(\t) \, d \t = \int_{\mathbb{R}^m} f_n(\t) \phi(\t) \, d \t, \quad \forall \phi \in C_c^\infty(\mathbb{R}^m)\, .
\end{equation}
Now, for every \(\phi \in C_c^\infty(\mathbb{R}^m)\),
\[
\bigl| u(\t + {\bf p}_n) \Delta \phi(\t) \bigr| \le \Bigl( \sup_{\t \in \mathbb{R}^m} |u(\t)| \Bigr) |\Delta \phi(\t)|
\]
and
\[
\bigl| f(\t + {\bf p}_n) \phi(\t) \bigr| \le \Bigl( \sup_{\t \in \mathbb{R}^m} |f(\t)| \Bigr) |\phi(\t)|.
\]
Both right-hand sides define integrable functions over \(\mathbb{R}^m\) (since $u$, $f$ are bounded and $\phi$ has compact support). Therefore, passing to the limit as \(n \to \infty\) in \eqref{EeQqNn0011} and applying the Lebesgue's Dominated Convergence Theorem, we obtain
\[
\int_{\mathbb{R}^m} h(\t) \Delta \phi(\t) \, d \t = \int_{\mathbb{R}^m} g(\t) \phi(\t) \, d \t, \quad \forall \phi \in C_c^\infty(\mathbb{R}^m).
\]
Hence, \(\Delta h = g\) in the distributional sense. An analogous argument shows that \(\Delta k = f\) in the distributional sense. Therefore, in the distributional sense, we have
\[
\Delta (k - u) = 0\, .
\]
On the other hand we have that \(k - u \in BUC(\mathbb{R}^m ; \mathbb{R}) \subset L^1_{\mathrm{loc}}(\mathbb{R}^m)\) then, from  Weyl's results (Theorem  \ref{Thmf01} part $1$) we obtain that \(k - u\) is smooth and harmonic. Also, since \(k - u \) is bounded, Lioville's results (Theorem \ref{Thmf01} part $2$), implies that there exists a constant \(c \in \mathbb{R}\) such that
\[
k = u + c\, .
\]

It remains to prove that \(c = 0\). First note that, since \(u(\t + {\bf p}_n) \le \sup u\) and \(u(\t+p_n ) \to h(\t) \) 
for all \(\t \in \mathbb{R}^m\), then  \(\sup h \le \sup u\). Similarly, \(\sup k \le \sup h\), consequently
\[
\sup k \le \sup u \, .
\]
Also, \(\inf u \le \inf h \le \inf k\). But \(\inf k = \inf u + c\), then \(c \ge 0\).

Also, since 
\(
\sup k = \sup u + c \), then \( c = \sup k - \sup u \le 0\, .
\) 
Therefore, \(c = 0\), which concludes that \(k = u\).\\
\end{proof}

As immediate consequences of Theorem \ref{main theorem -1}, we obtain the compact almost automorphic regularity of Poisson's equation under weak recurrent notions (compared with the almost automorphy) in the source function $f$. The results are presented in the next corollaries

\begin{cor}\label{main theorem} 
 Let the function  $f:\mathbb{R}^m \to \mathbb{R}$ satisfy one of the following conditions
\begin{enumerate}
\item $f\in \mathbb{Z}^m AA(\mathbb{R}^{m},\mathbb{R}) \cap BC(\mathbb{R}^{m},\mathbb{R}) $.
\item $f\in \mathcal{G} AA(\mathbb{R}^{m},\mathbb{R}) \cap BC(\mathbb{R}^{m},\mathbb{R})$, where $\mathcal{G}$ is a dense subset of $\mathbb{R}^m$.
\item $f\in \mathcal{G} AA(\mathbb{R}^{m},\mathbb{R})$, where $\mathcal{G}$ is a dense subgroup of $\mathbb{R}^m$.
\end{enumerate}

If  $u:\mathbb{R}^{m}\to \mathbb{R}$ is bounded, continuous and solution in the sense of distributions of Poisson's equation
\begin{equation}
\notag\Delta u=f\, ,
\end{equation}
then $u\in \mathcal{K}AA(\mathbb{R}^{m},\mathbb{R})$.
\end{cor}
\begin{proof}

If \( f \) satisfies any of the conditions \(1\), \(2\), or \(3\), then it is continuous and bounded. Since \( u \) is also continuous and bounded, the Lipschitz regularity of distributional solutions to Poisson’s equation implies that \( u \) is globally Lipschitz on \( \mathbb{R}^m \), and hence uniformly continuous on \( \mathbb{R}^m \).

Therefore, from the characterization theorems for compact almost automorphic functions presented in Section~\ref{section char}, it remains to show that \( u \) is \( \mathcal{G} \)-multi-almost automorphic, where \( \mathcal{G} = \mathbb{Z}^m \) in case \(1\), a dense subset of \( \mathbb{R}^m \) in case \(2\), or a dense subgroup of \( \mathbb{R}^m \) in case \(3\).  Since this property does not depend on the algebraic structure or the denseness of \( \mathcal{G} \), the proof can be carried out in a unified and abstract manner, and it is completely analogous to that of Theorem \ref{main theorem -1}.

\end{proof}

%

\begin{cor}\label{classic KAA reg Poisson} Let the function  $f:\mathbb{R}^m \to \mathbb{R}$ satisfy one of the following conditions
\begin{enumerate}
\item  $f\in AA(\mathbb{R}^{m} ; \mathbb{R})$.
\item  $f\in KAA(\mathbb{R}^{m} ; \mathbb{R})$.
\end{enumerate}
If  $u:\mathbb{R}^{m}\to \mathbb{R}$ is bounded, continuous and solution in the sense of distributions of Poisson's equation
\begin{equation}
    \notag\Delta u=f
\end{equation}
then, $u\in KAA(\mathbb{R}^{m} ; \mathbb{R})$.
\end{cor}


The next corollary is the almost periodic regularity result proved in \cite{sibuya1971almost}.
\begin{cor}(Sibuya's Theorem) Let $f\in AP(\mathbb{R}^{m},\mathbb{R})$. If $u$ is a continuous function, bounded and solution in the sense of distributions of:
\begin{equation}
    \notag\Delta u=f
\end{equation}
then, $u\in AP(\mathbb{R}^{m},\mathbb{R})$.
\end{cor}


\subsection{Compact almost automorphy of the solution of heat equation}\label{Sect:55}

Let us consider the initial value problem for the heat equation
\begin{equation}\label{Heat equation}
            \left\{
                 \begin{array}{lll}
                \partial_t u(t,x)&=&\Delta u(t,x), \quad t>0,\, x\in \mathbb{R}^m,\\
                   u(0,x) &= &f(x), \quad x\in \mathbb{R}^m,
                 \end{array}\right.
 \end{equation}
where $ f \in BC(\mathbb{R}^m ; \mathbb{R}) $, $\partial_t \cdot := \dfrac{\partial \cdot}{\partial t}$ is the partial derivative with respect to $t$ and \( \Delta \) denotes the Laplace operator in the spatial variable \( x = (x_1, \ldots, x_m) \), defined as \( \Delta := \sum_{i=1}^m \partial^2_{x_i} \). The maximal domain of the Laplace operator $ \Delta $ in the Banach space $Z= BC(\mathbb{R}^m ; \mathbb{R}) $, is given by $$ D(\Delta):=\lbrace v\in Z : \Delta v \text{ exists in } Z\rbrace  $$
which is not necessarily a usual functional space and is not densely defined in $Z$, i.e., $ \overline{D(\Delta)}\subsetneq Z $, where the last closure uses the supnorm.

\begin{definition}
Let $ f\in BC(\mathbb{R}^m ; \mathbb{R}) \cap \overline{D(\Delta)}$, by a solution of equation \eqref{Heat equation} we mean a function $u: [0,+\infty)\times\mathbb{R}^m \longrightarrow \mathbb{R}$, such that $u \in C^{1}((0,+\infty) ; BC(\mathbb{R}^m ; \mathbb{R}))$, $ u(t,\cdot) \in D(\Delta)$ for all $t>0$ and $u$ satisfies \eqref{Heat equation} pointwise.\\

In particular, if $f \in BUC(\mathbb{R}^m ; \mathbb{R}) \cap \overline{D(\Delta)}  $, the solution $u$ satisfies $u \in C([0,+\infty), BC(\mathbb{R}^m ; \mathbb{R}))\cap C^{1}((0,+\infty) ; BC(\mathbb{R}^m ; \mathbb{R}))$, $ u(t,\cdot) \in D(\Delta) $ for all $t>0$ and $u$ satisfies \eqref{Heat equation} pointwisely.
\end{definition}

Let $ (T(t))_{t\geq 0} $ be the Gaussian semigroup in $ BC(\mathbb{R}^m ; \mathbb{R}) $. It is well-known from semigroup theory, that the heat equation \eqref{Heat equation} admits a solution $u$ which  is given by 
$$ u(t,x)=T(t)f(x), \quad t>0, \, x\in \mathbb{R}^m, \quad \text{ for } f\in \overline{D(\Delta)}  $$
and $ u(0,x)=f(x) $. From \cite[Example 3.7.6]{Arendt}, we know that $ (T(t))_{t\geq 0} $ is given by the following convolution formula:
\begin{eqnarray}\label{Gaussian semigroup}
 T(t)f(x)&=& \dfrac{1}{(4\pi t)^{m/2}} \int_{\mathbb{R}^m} e^{-\|x-z\|^2 /(4t)} f(z) dz= \dfrac{1}{(4\pi t)^{m/2}} \int_{\mathbb{R}^m} e^{-\|z\|^2 /(4t)} f(x-z) dz \nonumber \\
 &=&(K(t,\cdot)*f)(x), \quad x\in \mathbb{R}^m\, .
\end{eqnarray}
where
$$K(t,x):=\dfrac{1}{(4\pi t)^{m/2}} e^{-\|x\|^2 /(4t)}\, .$$
%
\begin{rem}\label{Remark Semigroup} Note that the Gaussian semigroup defined by \eqref{Gaussian semigroup} is also well defined in the space $ L^{\infty}(\mathbb{R}^m ; \mathbb{R}) $, and for all $t>0$, we have $T(t) L^\infty(\mathbb{R}^m ; \mathbb{R}) \subset BUC(\mathbb{R}^m ; \mathbb{R})$ (see for instance the proof of the next Theorem).  
\end{rem}

\begin{tw}\label{Gaussian01}
For the Gaussian semigroup we have
 $$T(t)\big{(} {\mathrm{R}}_{{\mathcal{G}}}AA(\mathbb{R}^m ; \mathbb{R}) \cap BC(\mathbb{R}^m ; \mathbb{R}) \big{)} \subset \mathcal{K} {\mathrm{R}}_{{\mathcal{G}}}AA(\mathbb{R}^m ; \mathbb{R}) \, , \, \, \, \forall t >0\, .$$
\end{tw}
\begin{proof}
Let $f\in {\mathrm{R}}_{{\mathcal{G}}}AA(\mathbb{R}^m ; \mathbb{R}) \cap  BC(\mathbb{R}^m ; \mathbb{R})$, then according to Theorem \ref{TI0101}, we have that for each $t>0$, $T(t)f(\cdot)$ belongs to the space ${\mathrm{R}}_{{\mathcal{G}}}AA(\mathbb{R}^m ; \mathbb{R})$. Let us prove that for all $t>0$, $u(t,\cdot)$ is globally uniformly continuous in space variable and the conclusion will follow from Theorem \ref{ThRKAA Partial converse}.

Let $x,y\in \mathbb{R}^m$; then 
\begin{align}\label{Lipch}
|T(t)f(x)-T(t)f(y) | \leq \dfrac{1}{(4\pi t)^{m/2}} \int_{\mathbb{R}^m} |e^{-\|x-z\|^2 /(4t)}-e^{-\|y-z\|^2 /(4t)} |dz \, \|f\|_{\infty}.
\end{align}
Also, for fixed $t>0$ and $z\in \mathbb{R}^m$, the function $x\longmapsto e^{-\|x-z\|^2 /(4t)}$ is globally Lipschitz continuous, i.e., 
$$| e^{-\|x-z\|^2 /(4t)}-e^{-\|y-z\|^2 /(4t)}| \leq \sup_{w\in \mathbb{R}^m} \dfrac{c_m}{2t} \| w-z \|e^{-\|w-z\|^2 /(4t)} \| x-y\|, $$
where $c_m > 0$ is a constant. Note that
$$\sup_{w\in \mathbb{R}^m} \| w-z \|e^{-\|w-z\|^2 /(4t)}=\| w_0-z \|e^{-\|w_0-z\|^2 /(4t)}\, \, , \,\,\,  w_0 \in \mathcal{D}\, ,$$
where 
$$\mathcal{D}:=\{ w\in \mathbb{R}^m\, : \, \|w-z \| =\sqrt{2t}\}\, .$$
Now, from (\ref{Lipch}) and using invariance under translation of the Lebesgue integral, we have
\begin{align*}
|T(t)f(x)-T(t)f(y) | &\leq c'_m(t) \|f\|_{\infty} \|x-y\|\, ,
\end{align*}
where, $c_m^{'}(t)$ is an appropriated constant depending on $t$. Hence, $x\longmapsto T(t)f(x) $ is uniformly continuous.
\end{proof}

Since the Gaussian semigroup is a $C_0$ semigroup in the space $BC(\mathbb{R}^m ; \mathbb{R}) \cap \overline{D(\Delta)}$, then it is immediate the next Theorem 
\begin{tw}
Let $f\in  {\mathrm{R}}_{{\mathcal{G}}}AA(\mathbb{R}^m ; \mathbb{R}) \cap BC(\mathbb{R}^m ; \mathbb{R}) \cap \overline{D(\Delta)}$. Then, the unique bounded solution $u$ of the heat equation \eqref{Heat equation} is such that, for all $t>0$,  $u(t,\cdot) \in  \mathcal{K} {\mathrm{R}}_{{\mathcal{G}}}AA(\mathbb{R}^m ; \mathbb{R})$.
\end{tw}
 We understood the previous result as the compact-${\mathrm{R}}_{{\mathcal{G}}}$-multi-almost automorphic regularity, in space variable, of the heat equation.

The next result is a consequence of Theorem \ref{Gaussian01}, the characterization theorems for compact almost automorphic functions presented in Section~\ref{section char} and of Theorem \ref{TI0101}

\begin{cor}\label{main theorem cor} 
 Let $t>0$, then the Gaussian semigroup satisfies
\begin{enumerate}
\item $T(t) \big{(} \mathbb{Z}^m AA(\mathbb{R}^{m} ; \mathbb{R}) \cap BC(\mathbb{R}^{m}; \mathbb{R}) \big{)}\subset \mathcal{K} AA(\mathbb{R}^m ; \mathbb{R})$.
\item $T(t) \big{(}  \mathcal{G} AA(\mathbb{R}^{m} ; \mathbb{R}) \cap BC(\mathbb{R}^{m} ; \mathbb{R}) \big{)}\subset \mathcal{K} AA(\mathbb{R}^m ; \mathbb{R})$, where $\mathcal{G}$ is a dense subset of $\mathbb{R}^m$.
\item $T(t)  \mathcal{G} AA(\mathbb{R}^{m} ; \mathbb{R})  \subset \mathcal{K} AA(\mathbb{R}^m ; \mathbb{R})$, where $\mathcal{G}$ is a dense subgroup of $\mathbb{R}^m$.
\end{enumerate}
\end{cor}

\begin{cor}
Let the function  $f:\mathbb{R}^m \to \mathbb{R}$ satisfy one of the following conditions
\begin{enumerate}
\item $f\in \mathbb{Z}^m AA(\mathbb{R}^{m} ; \mathbb{R})\cap BC(\mathbb{R}^m ; \mathbb{R}) \cap \overline{D(\Delta)} $.
\item $f\in \mathcal{G} AA(\mathbb{R}^{m} ; \mathbb{R})\cap BC(\mathbb{R}^m ; \mathbb{R}) \cap \overline{D(\Delta)} $, where $\mathcal{G}$ is a dense subset of $\mathbb{R}^m$.
\item $f\in \mathcal{G} AA(\mathbb{R}^{m} ; \mathbb{R})\cap BC(\mathbb{R}^m ; \mathbb{R}) \cap \overline{D(\Delta)} $, where $\mathcal{G}$ is a dense subgroup of $\mathbb{R}^m$.
\end{enumerate}
Then, the unique bounded solution $u$ of the heat equation \eqref{Heat equation} is such that, for all $t>0$, $u(t,\cdot) \in  \mathcal{K} AA(\mathbb{R}^m ; \mathbb{R})$.
\end{cor}


\begin{cor}\label{Corollary KAA Heat}  Let the function  $f:\mathbb{R}^m \to \mathbb{R}$ satisfy one of the following conditions
\begin{enumerate}
\item  $f\in  AA(\mathbb{R}^m ; \mathbb{R})\cap BC(\mathbb{R}^m ; \mathbb{R}) \cap \overline{D(\Delta)}$.

\item $f\in \mathcal{K} AA(\mathbb{R}^m ; \mathbb{R}) \cap \overline{D(\Delta)}$.
\end{enumerate}

Then, the unique bounded solution $u$ of the heat equation \eqref{Heat equation} is such that, for all $t>0$, $u(t,\cdot) \in  \mathcal{K} AA(\mathbb{R}^m ; \mathbb{R})$.
%
%
\end{cor}

Finally, we also have the next almost periodic regularity result 

\begin{cor}
Let $f\in AP(\mathbb{R}^{m} ; \mathbb{R}) \cap \overline{D(\Delta)}$. Then, equation \eqref{Heat equation} admits a unique solution $u$ such that, for all $t>0$,  $u(t ,\cdot) \in AP(\mathbb{R}^m ; \mathbb{R})$.
\end{cor}

\appendix
\section{Sequentially almost automorphic functions on topological groups}
Let $\mathcal{G}$ be a topological group (commutative, locally compact). We will write the group operation on $\mathcal{G}$ additively, that is by $+$, the inverse operation by $-$, \color{blue} and the additive identity element by $0$ \color{black}.
\begin{definition}\label{Defi01}
Let $f: \mathcal{G} \to \mathcal{X}$ be a continuous function. $f$ is sequentially almost automorphic  if, for any sequence $( {\bf s}^{\prime}_{n})$ of $\mathcal{G}$ there exist a subsequence $( {\bf s}_{n})$ of $({\bf s}^{\prime}_{n})$ and a function $f^*:\mathcal{G}\to \mathcal{X}$ such that the following pointwise limits holds:
\begin{equation}
\notag\lim_{n\to \infty}f(\t+{\bf s}_{n})=f^*(\t),
\end{equation}
and
\begin{equation}
\notag\lim_{n\to \infty}f^*(\t-{\bf s}_{n})=f(\t)\, .
\end{equation}
\end{definition}
\noindent $AA(\G ; \X)$ will denote the space of sequentially almost automorphic functions. The next two theorems are of importance in the present work.

\begin{tw}\label{pro1}
If $f \in AA(\G ; \X)$, then $f$ is bounded.
\end{tw}
\begin{proof}

Suppose that $f$ is not bounded, that is, there exists a sequence $({\bf s}'_n)_{n\in \mathbb{N}} \subset \G$ such that
\begin{equation}\label{Apendix 01}
    \lim_{n\rightarrow\infty}{||f({\bf s}'_n)||_{\mathcal{X}}} = \infty\, .
\end{equation}
Since $f\in AA(\G ; \mathcal{X})$, there exist a subsequence $({\bf s}_n)_{n\in \mathbb{N}}\subset({\bf s}'_n)_{n\in \mathbb{N}}$ and a function $f^*:\G \rightarrow \mathcal{X}$ such that, for each $\t \in \mathcal{G}$,  the following limits holds
\begin{equation*}
    f^*(\mathbf{t}) =  \lim_{n\rightarrow \infty}f(\mathbf{t}+{\bf s}_n)\,\,, \,\,
f(\mathbf{t}) =  \lim_{n\rightarrow \infty}f^*(\mathbf{t}-{\bf s}_n)\, .
\end{equation*}
Now, if we take $\textbf{t} = 0$ in the previous limit, we have
\begin{equation*}
    \lim_{n\rightarrow\infty}{f({\bf s}_{n})} = f^*(0)\,\,; \,\, f^*(0) \in \mathcal{X},
\end{equation*}
which implies the next contradiction with (\ref{Apendix 01})
\begin{equation*}
    \lim_{m\rightarrow \infty}{||f({\bf s}_{n})||_{\mathcal{X}}} = ||f^*(0)||_{\mathcal{X}} < \infty\, .
\end{equation*}
Therefore, $f$ is bounded . 
\end{proof}
\begin{tw}\label{pro2}
If $f \in AA(\G ; \X)$, then the range of $f$ is relatively compact in $\X$.
\end{tw}
\begin{proof}
    
Let $(y'_n)_{n\in \mathbb{N}} \subset \mathrm{R}(f) := \{y \in \X: \exists \t \in \G \,\,\text{whit}\,\, f(\t) = y\}$ be an arbitrary sequence. Then, there exists a sequence $({\bf s}'_n)_{n\in \N}\subset \G$ such that $y'_n = f({\bf s}'_n)$. On the other hand, since  $f \in AA(\G ; \X)$ there exist a subsequence $({\bf s}_{n})_{n\in \N} \subset ({\bf s}'_n)_{n\in \N}$ and a function $f^*:\G \rightarrow \X$ such that 
\begin{equation*}
     \lim_{n\rightarrow\infty}{f({\bf s}_{n})} = f^*(0)\, .
\end{equation*}
Therefore, there exists a subsequence $(y_n)$ of $(y'_n)$, defined by $y_n:=f({\bf s}_{n})$, which converge in $\X$. 
\end{proof}


\section{Subgroups of $\mathbb{R}^m$}
Here we summarize important results on dense subgroups of the group $\mathbb{R}^m$. We start with the following assertion:
\begin{lem}\label{lem8}
Every nontrivial subgroup of $\mathbb{R}^m$ is unbounded.
\end{lem}
\begin{proof}
    Let $ H $ be a nontrivial subgroup of $\mathbb{R}^m $. 
Since $H $ is nontrivial, it contains at least one element $v$ other than the zero vector. Now, for eack $k\in \mathbb{N}$, define  $ v_k := k \cdot v  \in H$, where $k \cdot v =v+v+\cdots +v $ (k-times). Then, for each $ k $, we have $ \|v_k\| = \|k \cdot v\| = k \cdot \|v\| $. Therefore, the sequence $ \{v_k\} $ is unbounded. This proves that $ H $ is unbounded.
\end{proof}
\subsection{Dense subgroups of $\mathbb{R}^m$}
Let $\mathbb{Q}$ be the rational numbers. Some dense subgroups of $\mathbb{R}^m$ are, for example, $\mathbb{Q}^m$, $\mathbb{Q}^k\times \mathbb{R}^{m-k}$, $\mathbb{Q}^{k_1}\times G^{k_2} \times \mathbb{R}^{k_3}$ where $k_1+k_2+k_3=m$ and $G^{k_2}$ is a dense subgroup of $\mathbb{R}^{k_2}$, etc. In the following Theorem and in the subsequent Proposition, it is described some conditions under which a subgroup of $\mathbb{R}^m$ becomes dense. 

\begin{tw}\label{ApB1}
    Let $\theta_1, \theta_2, \cdots,\theta_m$ real numbers.  For the subgroup 
    $$\mathbb{Z}^{m} + \mathbb{Z}(\theta_1, \theta_2, \cdots,\theta_m)= \{(z_1 + z_0\theta_1, \cdots,z_m+z_0\theta_m); \,\,(z_0,z_1,\cdots,z_m)\in \mathbb{Z}^{m+1}\}\subset \mathbb{R}^m $$ 
    to be dense in $\mathbb{R}^m$, it is necessary and sufficient that the $m+1$ numbers $1,\theta_1, \theta_2, \cdots,\theta_m$ be linearly independent on $\mathbb{Q}$.
\end{tw}
\begin{prop}\label{ApB2}
 Let $G$ be a finitely generated subgroup of $\mathbb{R}^m$. The following conditions are equivalent:  
 \begin{enumerate}
     \item $G$ is dense in $\mathbb{R}^m$.
     \item For any nonzero linear form $\phi: \mathbb{R}^m\rightarrow\mathbb{R}$ we have $\phi(G)\not\subset\mathbb{Z}$.
 \end{enumerate}
\end{prop}
For a proof of Theorem \ref{ApB1} and Proposition \ref{ApB2}, the reader is invited to consult the works 
\cite{elghaoui2015rational,waldschmidt1995topologie}.
%
%

\section{A property of continuous functions}
\begin{prop}\label{grp}
    Let $f:\mathbb{R}^m \rightarrow \X$ be a continuous function and $A \subset \mathbb{R}^m$, then $$\overline{f(\overline{A})} = \overline{f(A)}\, .$$
\end{prop}

\begin{proof}
First note that if $A$ is empty, then the conclusion holds immediately. In case $A$ is not empty, we prove the equality by checking the double inclusion separately. 

\bigskip
\noindent\textbf{1. } $\overline{f(A)} \subseteq \overline{f(\overline{A})}$

This inclusion holds for any function $f$ (continuity is not needed here). In fact, since $A \subseteq \overline{A}$, we have
\[
f(A) \subseteq f(\overline{A}).
\]
Now, taking the closure of both sides and since the closure operation is monotone under inclusions, we obtain
\[
\overline{f(A)} \subseteq \overline{f(\overline{A})}.
\]

\bigskip
\noindent\textbf{2. } $\overline{f(\overline{A})} \subseteq \overline{f(A)}$.

Let us prove the following inclusion first $f(\overline{A}) \subseteq \overline{f(A)}$. Let $z$ be an element of $f(\overline{A})$ then, there exists an element $x \in \overline{A}$ such that $f(x)=z$. Also, there exists a sequence $(x_n)_{n\in \mathbb{N}}$ in $A$ such that $x_n \to x$ as $n \to \infty$. Since $f$ is continuous, we have $z_n:=f(x_n)\to f(x)=z$ as $n\to \infty$. That is, there exists a sequence $(z_n)_{n\in \mathbb{N}}$ in $f(A)$ such that $z_n \to z$ as $n \to \infty$, which implies that $z \in \overline{f(A)}$. Therefore, we conclude 
$$f(\overline{A}) \subseteq \overline{f(A)}\, .$$
Taking again the closure of both sides, yields
$$\overline{f(\overline{A})} \subseteq \overline{f(A)}\, .$$
Finally, Combining both inclusions we obtain
\[
\overline{f(\overline{A})} = \overline{f(A)}\, .
\qedhere
\]
\end{proof}

\subsection*{Acknowledgments}
The authors would like to express their gratitude to the anonymous referees for their careful reading and helpful comments.

\subsection*{Funding}
A. Ch\'avez and J. Casta\~neda were partially supported by CONCYTEC through the PROCIENCIA program under the E041-2023-01 competition, according to contract PE501082885-2023. A. Ch\'avez, J. Casta\~neda and A. Rodr\'iguez were also partially supported by the National Council for Science, Technology and Technological Innovation (CONCYTEC) and by the National Program for Scientific Research and Advanced Studies (PROCIENCIA) within the framework of the call “E041-2025-01 Basic Research Projects” under Contract No. PE501094767-2025-PROCIENCIA


\section*{Declarations}

\subsection*{Conflict of interest}
No potential conflict of interest was reported by the authors.



\end{document}